\documentclass[12pt]{iopart}

\expandafter\let\csname equation*\endcsname\relax
\expandafter\let\csname endequation*\endcsname\relax
\usepackage{amsmath,amssymb,amsthm,amsfonts,amstext,amsbsy,amscd}
\usepackage{graphicx}
\usepackage{enumitem}
\usepackage{mathabx}			
\usepackage{subfigure}
\usepackage[numbers,nonamebreak]{natbib}
\usepackage[colorlinks,,linkcolor=red,citecolor=blue,urlcolor=red]{hyperref}
\usepackage{multirow}%
\usepackage{slashbox}
\usepackage{booktabs}
\newcommand{\bea}{$$ \begin{array}{lll}}
\newcommand{\eea}{\end{array} $$}
\newcommand{\ei}{\end{itemize}}

\newtheorem{satz}{Satz}[section]

\newtheorem{theorem}[satz]{Theorem}
\newtheorem{proposition}[satz]{Proposition}

\newtheorem{lemma}[satz]{Lemma}
\newtheorem{assumption}[satz]{Assumption}

\newtheorem{remark}[satz]{Remark}

\let\epsilon=\varepsilon
\let\ep=\epsilon
\let\phi=\varphi
   
\let\de=\delta

\newcommand{\ind}[1]{\boldsymbol{1}^{}_{ \left\{#1\right\}}}
\newcommand{\indi}[1]{\boldsymbol{1}^{}_{ #1}}
\newcommand{\Sum}{\displaystyle \sum}

\newcommand{\AlgI}{\textbf{Algorithm I}}
\newcommand{\AlgII}{\textbf{Algorithm II}}
\newcommand{\E}{\mathbb{E}}
\newcommand{\R}{\mathbb{R}}
\newcommand{\PP}{\mathbb{P}}

\begin{document}
\title{Noisy Laplace deconvolution with error in the operator}
\author{Thomas Vareschi}
\address{Universit\'e Denis Diderot Paris 7, B\^atiment Sophie-Germain, rue Alice-Domon et L\'eonie-Duquet, 75013 Paris, France}
\ead{thomas.vareschi@univ-paris-diderot.fr}
\begin{abstract}
We adress the problem of Laplace deconvolution with random noise in a regression framework. The time set is not considered to be fixed, but grows with the number of observation points. Moreover, the convolution kernel is unknown, and accessible only through experimental noise. We make use of a recent procedure of estimation based on a Galerkin projection of the operator on Laguerre functions (\cite{CPR}), and couple it with a threshold performed both on the operator and the observed signal. We establish the minimax optimality of our procedure under the squared loss error, when the smoothness of the signal is measured in a Laguerre-Sobolev sense and the kernel satisfies fair blurring assumptions. It is important to stress that the resulting process is adaptive with regard both to the target function's smoothness and to the kernel's blurring properties. We end this paper with a numerical study emphazising the  good practical performances of the procedure on concrete examples.
\end{abstract}


\section{Introduction}

Laplace deconvolution is motivated by a wide set of practical applications, ranging from population dynamics or physics to computational tomography or fluorescence spectroscopy (\citet[Chap. 2]{Linz}, \citet{AB}, \citet{CPR}). In the corresponding setting we observe $\boldsymbol{q}$, the result of the action of a kernel $\boldsymbol{g}$ on the function of interest $\boldsymbol{f}$, according to the following equation
\begin{align}
\label{Laplace deconvolution det}
\boldsymbol{q}(t)=\int_{0}^t \boldsymbol{g}(t-\tau)\boldsymbol{f}(\tau)d\tau,\;t\geq 0
\end{align}
Equation \eqref{Laplace deconvolution det} is also refered to as Volterra integral equation. One of its main features is its causal property, since $\boldsymbol{q}(t)$ is affected only by the values of $\boldsymbol{f}$ and $\boldsymbol{g}$ at times anterior to $t$. Of course, only finite samples of $\boldsymbol{q}(t)$ are accessible in practice. Moreover, the presence of additional noise justifies the empirical modelization of \eqref{Laplace deconvolution det} by the classical regression model, inspired by \citet{APR}
\begin{align}
\label{Laplace deconvolution reg}
\boldsymbol{y}(t_i)=\int_0^{t_i} \boldsymbol{g}(t_i-\tau)\boldsymbol{f}(\tau)d\tau+\sigma \eta_i,\;i=1,...,n
\end{align}
where $0\leq t_1\leq ...\leq t_n \leq T_n$ are the points of observation, $(\eta_i)_{i=1,...,n}$ are independent standard gaussian variables, and $\sigma$ is a fixed factor accounting for the precision of the observations. $T_n$ is supposed to grow with the number of observations $n$.
\\As pointed out in \citet{APR} and \citet{CPR}, in spite of its apparent similarity with the Fourier deconvolution problem, the theoritical features of equation \eqref{Laplace deconvolution det}, as well as the practical problems raised during its resolution are deeply different. More precisely, setting artificially $\boldsymbol{g}(t)=\boldsymbol{f}(t)=0$ for $t<0$ amounts to solving the classical Fourier deconvolution problem
\begin{align}
\boldsymbol{y}(t_i)=\int_0^{T_n} \boldsymbol{g}(t_i-\tau)\boldsymbol{f}(\tau)d\tau+\sigma \eta_i,\;i=1,...,n
\end{align}
A first notable objection is that the framework of classical Fourier deconvolution assumes periodicity of the function $f$ and the kernel $g$ on $[0,T]$, a meaningless notion when applied to a varying time set $[0,T_n]$. Even more problematic is the fact that this modelization totally ignores the causal feature of Laplace convolution, creating unwanted interferences between different time sets. To finish, the manipulation consisting in artificially expanding $\boldsymbol{q}$ and $\boldsymbol{g}$ for $t<0$ creates artifacts on the estimated function at times $t<0$ as well.
\\Another approach is to treat equation \eqref{Laplace deconvolution reg} as a general ill-posed problem and apply a Tikhonov regularization (\citet{Golubev}). However the direct implementation of this method also destroys the causal nature of equation \eqref{Laplace deconvolution det}, and tends to oversmooth the solution (\citet{CL}). Subsequent adaptations which remedy these shortcomings are present in \citet{Lamm2} and \citet{CL}. However in these works the time set is considered to be fixed.
\\ A more suitable theoritical tool in solving \eqref{Laplace deconvolution det} is the use of Laplace transform, which allows to derive a closed form of the solution. However, its direct implementation is compromised by numerical problems, since the generic expression of the inverse Laplace transform is not easily computable in general. This motivates the widespread use of inversion tables, unfortunately irrelevant when the image function is not known exactly but approximated via a numerical scheme.
\\ In this paper, following \citet{CPR}, we will exploit the properties of Laguerre functions, which can be used either to compute the inverse Laplace transform (\citet{ACW}, \citet{LTD}), or to solve directly equation \eqref{Laplace deconvolution det} (\citet{KN}). More precisely, a Galerkin method applied to \eqref{Laplace deconvolution reg} shows that, even if their role is not entirely symmetric to the role played by harmonics in the framework of Fourier deconvolution, they allow a sparse analysis of equation \eqref{Laplace deconvolution det}.
\\ All the previous mentionned works only concerned the case of a deterministic noise at best. The presence of random noise requires an additionnal treatment, and calls for specific statistical tools. In the setting of random noise, \citet{DMR} considered a kernel of the form $e^{-at}$ and used a regularized inversion of the inverse Laplace transform. More recently, \citet{APR} conceived an optimal procedure in the minimax sense on H\"older spaces $\boldsymbol{H}^s(\R_+)$. This procedure used an exact expression of the solution involving the derivatives of $\boldsymbol{q}$, which were then estimated via Lepskii's method. However a shortcoming of the procedure is its strong dependence on the kernel $\boldsymbol{g}$, in the sense that a small error in $\boldsymbol{g}$ can translate into a wide difference in the result. In other words there seems to be a trade off between the closed form of the solution, and the unstability with regard to the kernel. Moreover, the fact that $\boldsymbol{g}$ is seldom observed directly in practice, but is usually subject to experimental noise should prompt us to privilege stability over exactitude.
\\In that spirit, \citet{CPR} took advantage of the algebraic properties of Laguerre functions in the context of \eqref{Laplace deconvolution reg}. With an adequate penalty term, they proposed an estimator which mimicks the oracle risk to within logarithmic terms. This modelization has the non negligible advantage of practical simplicity and efficiency, since solving equation \eqref{Laplace deconvolution det} amounts to the inversion of a lower triangular Toeplitz matrix.
\\ Even if this latter procedure proves to be more stable with regard to $\boldsymbol{g}$ experimentally, no systematic study has been conducted on the subject yet. In this paper we attempt to fill in this gap: we suppose that the observation of $\boldsymbol{g}$ is contaminated by a gaussian white noise, and show how Laguerre functions allow to handle this issue. We place ourselves under the minimax point of view and suppose that $\boldsymbol{f}$ belongs to a Laguerre-Sobolev space and that $\boldsymbol{g}$ satisfies standard blurring assumptions. We apply recent techniques for the treatment of noisy operators in the context of inverse problems (\citet{HR},\citet{DHPV}), which consist in a preliminary processing of the operator $\boldsymbol{K}$ coupled with a classical thresholding procedure applied to $\boldsymbol{y}$.

\section{Discretization of Laplace deconvolution}

\subsection{Laguerre functions}
Suppose that the target function $\boldsymbol{f}$ and the kernel $\boldsymbol{K}$ both lie in $L^2(\R_+)$. Define the Laguerre  polynomials (see \citet{GR})
\begin{align}
L_\ell(t)=\Sum_{j=0}^\ell(-1)^j \binom{\ell}{j} \frac{t^j}{j!}
\end{align} and, following \citet{CPR}, the ensuing Laguerre functions, depending on the parameter $a>0$,
\begin{align}
\boldsymbol{\phi}_\ell(t)= \sqrt{2a} e^{-at} L_\ell(2at),\;\ell \in\mathbb{N}
\end{align}
The parameter $a$ is a tuning parameter used to fit experimental curves. The Laguerre functions constitute a Hilbert basis of $L^2(\R_+)$. Any function $f\in L^2(\R_+)$ satisfies
\begin{align}
f=\Sum_{\ell \geq 0} \check{f}_\ell\boldsymbol{\phi}_\ell(t),\;\Check{f}_\ell\overset{\Delta}{=}\int_0^\infty f(\tau) \boldsymbol{\phi}_\ell(\tau)d\tau
\end{align}
The following proposition illustrates the conveniency of Laguerre functions in the framework of equation \eqref{Laplace deconvolution det}.
\begin{proposition}[\citet{GR}, Formula 7.411.4]
\label{Laguerre Laplace}
\begin{align}
\forall a>0,\,\forall t\geq 0,\;\int_0^t \boldsymbol{\phi}_k(x)\boldsymbol{\phi}_\ell(t-x)dx=(2a)^{-1/2}\big(\boldsymbol{\phi}_{\ell+m}(t)-\boldsymbol{\phi}_{\ell+m+1}(t)\big)
\end{align}
\end{proposition}
From now on, except if explicitly mentionned, we will suppose $a=\frac{1}{2}$.
\subsection{Galerkin method}
\label{Galerkin section}
Proposition \ref{Laguerre Laplace} prompted \citet{CPR} to apply a Galerkin scheme to equation \eqref{Laplace deconvolution det}. Galerkin schemes rely on the choice of a set of functions which discretize the inverse problem at stake in a convenient way. They were beneficially  applied in the context of inverse problems (\citet{CHR}), and blind deconvolution (\citet{EK}, \citet{HR} and \citet{DHPV}). To this end we will remind briefly the underlying methodology of a Galerkin scheme and show how it conveniently applies to equation \eqref{Laplace deconvolution det}.
\\ Let $f\in L^2(\R_+)$ and $K$ an operator of $L^2(\R_+)$, and suppose we want to recover $f$ from the observation $q=Kf$. Note $\boldsymbol{V}_\ell$ the finite dimensional space spanned by the orthogonal set of Laguerre functions $\{\boldsymbol{\phi}_k\}_{k\leq \ell}$. The Galerkin approximation $f^\ell$ of $f$ on $\boldsymbol{V}_\ell$ is the solution of the equation
\begin{align}
\notag\langle Kf^\ell,v\rangle&=\langle g,v\rangle ,\; \forall v\in \boldsymbol{V}_\ell
\\\label{Galerkin}\Leftrightarrow \Sum_{k\leq \ell}\langle K\boldsymbol{\phi}_k,\boldsymbol{\phi}_k'\rangle\, \langle f^\ell, \boldsymbol{\phi}_k \rangle& = \langle g, \boldsymbol{\phi}_k \rangle ,\; \forall k'\leq \ell
\end{align}
We shall note $K^\ell$ the Galerkin matrix $(K^\ell)_{i,j}=\langle K\boldsymbol{\phi}_j,\boldsymbol{\phi}_i\rangle$, $i,j \leq \ell$. Note hence $\boldsymbol{K}$ the operator of $L^2(\R_+)$ mapping $f$ onto $t\mapsto \int_0^t f(t-\tau)g(\tau) d\tau$. We can reformulate \eqref{Laplace deconvolution det} as
\begin{align}
\label{Galerkin det}
\boldsymbol{q}^\ell=\boldsymbol{K}^\ell\boldsymbol{f}^\ell
\end{align}
Moreover, Proposition \ref{Laguerre Laplace} implies: 

\begin{proposition}[\citet{CPR}, Lemma 1]
The Galerkin matrix $\boldsymbol{K}^\ell$ is lower triangular, Toeplitz. More precisely, note $\boldsymbol{\dot{g}}$ the function with Laguerre coefficients $$\check{\dot{\boldsymbol{g}}}_\ell= \check{\boldsymbol{g}}_0\ind{\ell=0}+\big(\check{\boldsymbol{g}}_\ell-\check{\boldsymbol{g}}_{\ell-1}\big)\ind{\ell\geq 1},\;\;\ell \in \mathbb{N}$$ Then $$\boldsymbol{K}^\ell=\begin{pmatrix}
\check{\dot{\boldsymbol{g}}}_0&0&\hdots&0
\\\check{\dot{\boldsymbol{g}}}_1&\check{\dot{\boldsymbol{g}}}_0&\ddots&\vdots
\\\vdots&\ddots&\ddots&0
\\\check{\dot{\boldsymbol{g}}}_\ell&\hdots&\check{\dot{\boldsymbol{g}}}_1&\check{\dot{\boldsymbol{g}}}_0
\end{pmatrix}$$
\end{proposition}
In the sequel, for any function $f\in L^2(\R_+)$, we will note $T(f)$ the infinite Toeplitz matrix such that $T(f)_{i,1}=f_{i+1}$ for all $i\geq 0$, and $T_\ell(f)$ the extracted matrix defined by $T_\ell(f)_{i,1}=T(f)_{i,1}$, $i\leq \ell+1$. In particular,  $$\boldsymbol{K}^\ell=T_\ell(\boldsymbol{\dot{g}})$$
The resolution of the linear system \eqref{Galerkin} now shows great practical conveniency, provided that $\boldsymbol{K}^\ell$ is invertible. This is equivalent to $\check{\boldsymbol{g}}_0\neq 0$, an assumption we will make in the sequel. 
\subsection{Application to the regression model with irregular design}
It remains to incorporate two supplementary features of equation \eqref{Laplace deconvolution reg} in the inversion of \eqref{Galerkin det}. First, the presence of the random noise $\boldsymbol{\eta}$ and secondly, the possible irregularity of the design points. This construction is due to \citet{CPR}. Due to the fact that the observation points $t_i$ are imposed by the problem, the estimation of the Laguerre coefficients $\check{\boldsymbol{q}}_\ell$ of the function $\boldsymbol{q}$ suffers from two potential drawbacks. First, the infinite support of the Laguerre polynomials as well as the function $\boldsymbol{q}$ which should not be too problematic, provided that $T_n$ is large enough and that the functions decrease sufficiently to infinity. More problematic is the fact that the observation points $t_i$ are sometimes subject to experimental constraints, which affect their repartition on $\R_+$. The consistency of the estimation of $\check{\boldsymbol{q}}_\ell$ is hereby deteriorated.
\\ We will hence suppose that the following conditions are fulfilled:
\begin{itemize}
\item There exists an integer $n_0$ such that $\frac{n}{T_n}>\sigma$ for all $n\geq n_0$.
\item $\displaystyle{\lim_{n\to \infty}} T_n = \infty,\; \text{ and }\displaystyle{\lim_{n\to \infty}}\frac{T_n}{n} = 0$
\end{itemize}
To take into account the irregularity of the design, we follow \citet{CPR} and define $P_n:[0;T_n]\to[0;T_n]$ a regular non decreasing function such that 
\begin{align}
\label{Design map}
P_n(0)=0,\;P_n(T_n)=T_n,\;P_n(t_i)=\frac{i}{n}T_n \text{ for } i\leq n
\end{align}
Note $\boldsymbol{\Phi}_\ell$ the $(\ell+1)\times n$ matrix with entries $(\boldsymbol{\Phi}_\ell)_{k,i}=\boldsymbol{\phi}_k(t_i)$. For any function $h\in L^2(\R)$, we have 
\begin{align*}
\boldsymbol{P}_\ell h(t_i)=\Sum_{k\leq \ell} \boldsymbol{\phi}_k(t_i)\check{h}_k=\boldsymbol{\Phi}_\ell h^\ell
\\ \Leftrightarrow h^\ell=\big({^{\boldsymbol{t}}}\boldsymbol{\Phi}_\ell \boldsymbol{\Phi}_\ell\big)^{-1} {^{\boldsymbol{t}}}\boldsymbol{\Phi}_\ell \boldsymbol{P}_\ell h(t_i)
\end{align*}
where $\boldsymbol{P}_\ell$ is the orthogonal projector onto $\boldsymbol{V}_\ell$. We deduce that
\begin{align}
\label{Discretized model}
\boldsymbol{y}^\ell=\big({^{\boldsymbol{t}}}\boldsymbol{\Phi}_\ell \boldsymbol{\Phi}_\ell\big)^{-1} {^{\boldsymbol{t}}}\boldsymbol{\Phi}_\ell \boldsymbol{P}_\ell\big[\boldsymbol{q}+\sigma \boldsymbol{\eta}\big](t_i)=\boldsymbol{K}^\ell\boldsymbol{f}^\ell + \sigma \big({^{\boldsymbol{t}}}\boldsymbol{\Phi}_\ell \boldsymbol{\Phi}_\ell\big)^{-1} {^{\boldsymbol{t}}}\boldsymbol{\Phi}_\ell \boldsymbol{\eta}_n
\end{align}
where $\boldsymbol{\eta}_\ell\sim \mathcal{N}(0, \boldsymbol{I}_n)$. Let us take a closer look to the matrix $({^{\boldsymbol{t}}}\boldsymbol{\Phi}_\ell \boldsymbol{\Phi}_\ell\big)$. Its general term is
\begin{align*}
({^{\boldsymbol{t}}}\boldsymbol{\Phi}_\ell \boldsymbol{\Phi}_\ell)_{\ell,k}=&\Sum_{i=1}^n \boldsymbol{\phi}_k(P^{-1}(\frac{i}{n}T_n))\boldsymbol{\phi}_\ell(P^{-1}(\frac{i}{n}T_n))
\\\sim& \frac{n}{T_n}\int_{0}^{T_n}  \boldsymbol{\phi}_k(P^{-1}(\tau))\boldsymbol{\phi}_\ell(P^{-1}(\tau))d\tau
\\&=\frac{n}{T_n}\int_{0}^{T_n}  \boldsymbol{\phi}_k(\tau)\boldsymbol{\phi}_\ell(\tau)P'(\tau)d\tau
\end{align*}
for $n,T_n$ large enough. If the points $t_i$ are equispaced, taking $P(\tau)=\tau$ in \eqref{Design map} entails that $T_n n^{-1}({^{\boldsymbol{t}}}\boldsymbol{\Phi}_\ell \boldsymbol{\Phi}_\ell\big)$ is close to the identity provided that $T_n$ is large enough. As in \citet{CPR}, we hence reformulate \eqref{Discretized model} as the sequential model
$$\boldsymbol{y}^\ell=\boldsymbol{K}^\ell\boldsymbol{f}^\ell + \sigma \sqrt{\frac{T_n}{n}}\boldsymbol{\xi}_\ell$$
where $\boldsymbol{\xi}_\ell\sim \mathcal{N}(0,\boldsymbol{\Omega}_\ell)$ and $\boldsymbol{\Omega}_\ell=n T_n^{-1}({^{\boldsymbol{t}}}\boldsymbol{\Phi}_\ell \boldsymbol{\Phi}_\ell\big)^{-1}$. In general, $\boldsymbol{\Omega}_\ell$ somehow quantifies the distance to the uniform design case. To ensure that the design is not too ill conditionned, we will suppose that the following assumption is fulfilled.
\begin{assumption}
\label{Design regularity}
Let $L\in \mathbb{N}$. There exists $C\geq 0$, such that for all $\ell \leq L$, for all $\lambda\in \text{Sp}(\boldsymbol{\Omega}_\ell)$,
$\lambda \leq C$
\end{assumption}
This assumption is dependent on the integer $L$, which plays the role of a maximal resolution level, and will be adapted to the case of interest later. The inversion of \eqref{Discretized model} now requires controls of the variable $(\boldsymbol{K}^\ell)^{-1}\boldsymbol{\xi}_\ell$. Under suitable properties of $\boldsymbol{f}$ and $\boldsymbol{g}$, we shall be able to apply a classical inverse/thresholding procedure, and derive rates of convergence over specific regularity spaces. These properties are the subject of Part \ref{Regularity spaces}.

\subsection{Error in the operator}
We already mentionned the fact that the resolution of \eqref{Laplace deconvolution det} is usually unstable with respect to $\boldsymbol{g}$ (\citet{APR}). 
Furthermore, in practice, inference on the kernel $\boldsymbol{g}$ is possible only through experimental noise, and requires a preliminary step of estimation giving way to imprecision. This additionnal error might significantly contaminate the result of any procedure of estimation if not properly treated. Let us see how Laguerre functions $\boldsymbol{\phi}_\ell$ allow to handle this issue: in section \ref{Galerkin section}, we established that the discretization of \eqref{Discretized model with error} with Laguerre functions involved a Toeplitz matrix with entries constituted of the Laguerre coefficients of $\boldsymbol{\dot{g}}$. We can thus consider $\boldsymbol{\dot{g}}$ as the finite impulse response of the operator $\boldsymbol{K}$ when applied to the system $(\boldsymbol{\phi}_\ell)_{\ell\geq0}$. To take into account the imprecision in the observations of $\boldsymbol{\dot{g}}$, we adopt the framework of blind deconvolution and suppose that $\boldsymbol{\dot{g}}$ is not known exactly, but that we have acces the noisy version
\begin{align}
\label{Observation of g}
\boldsymbol{\dot{g}}_\de=\boldsymbol{\dot{g}}+\de \boldsymbol{{b}}
\end{align}
where $\boldsymbol{{b}}$ is a gaussian white noise on $L^2(\R_+)$. The generic problem of blind deconvolution is motivated by numerous scientific fields, including for example electronic microscopy or astrophysics, where the corresponding kernel is seldom known nor directly observed. It was adequatly discussed in \citet{EK} and \citet{HR}.
\\Taking into account the observations \eqref{Observation of g}, the projection $\boldsymbol{\dot{g}}^\ell$ is changed to $\boldsymbol{\dot{g}}_\de^\ell=\boldsymbol{\dot{g}}^\ell+\de \boldsymbol{{b}}^\ell$ where $\boldsymbol{{b}}^\ell$ is  a gaussian vector with covariance $\boldsymbol{I}_\ell$. The new model, adjusted from \eqref{Discretized model} becomes
\begin{align}
\label{Discretized model with error}
\begin{cases}\boldsymbol{y}^\ell&=\boldsymbol{K}^\ell \boldsymbol{f}^\ell+\sigma \sqrt{\frac{T_n}{n}} \boldsymbol{\xi}_\ell
\\\boldsymbol{K}_\de^\ell&= \boldsymbol{K}^\ell+\de \boldsymbol{{B}}^\ell
\end{cases}
\end{align}
where $\boldsymbol{{B}}^\ell=T_\ell(\boldsymbol{b})$ is a random Toeplitz matrix. In the sequel, for the sake of clarity, we note $\ep=\sigma \sqrt{\frac{T_n}{n}}$.
\begin{remark}
We could as well suppose that we observe $\boldsymbol{g}_\de=\boldsymbol{g}+\de\boldsymbol{b}$, yet it is more convenient to work with $\boldsymbol{\dot{g}}$ (the entries of the noisy Toeplitz matrix $\boldsymbol{B}$ are directly i.i.d standard gaussian variables). In the former case, the rest of the paper however adapts with no change in the algorithms, since inequality \eqref{ineq} is satisfied as well. A modification of the proof of Theorem \ref{Lower bound} should also provide the lower bound for the second procedure.
\end{remark}
\section{Features of the target function and the kernel}
\label{Regularity spaces}
\subsection{Sobolev spaces associated to Laguerre functions}
We proceed to the description of regularity spaces associated with the resolution of \eqref{Discretized model with error}. The following material is classical, we refer to \citet{BT} or \citet{Rath} for example.
\\ Since $f\mapsto \sqrt{2a}f(2a.)$ is an isometry of $L^2(\R_+)$, the structures defined for different values of $a$ are equivalent. Hence we shall only concentrate on the mainstream case where $a=1/2$.
Define the operator $\mathfrak{L}$ on $L^2(\R_+,dx)$ by
\begin{align}
\mathfrak{L}=-\Big[x\frac{d^2}{dx^2}+\frac{d}{dx}-\frac{x}{4}\Big]
\end{align}
The functions $\boldsymbol{\phi}_\ell$ are the eigenfunctions of $\mathfrak{L}$ associated with eigenvalues $(\ell+\frac{1}{2})$. We hence define the Sobolev space $\mathcal{W}^s$ associated with Laguerre functions as
\begin{align*}
\mathcal{W}^s=&\{f\in L^2(\R_+,dx) \text{ s.t. } \mathfrak{L}^s f \in L^2(\R_+,dx)\}
\end{align*}
For a function $f\in L^2(\R_+,dx)$, we have the straightforward equivalence
\begin{align*}
f\in \mathcal{W}^s&\Leftrightarrow\Sum_{\ell \geq 0} \big|\big(\ell+\frac{1}{2}\big)^s\langle f,\boldsymbol{\phi}_\ell\rangle \big|^2<\infty 
\end{align*}
and the associated norm 
$$\|f\|_{\mathcal{W}^s}=\Sum_{\ell \geq 0}\big|\big(\ell+\frac{1}{2}\big)^s\langle f,\boldsymbol{\phi}_\ell\rangle \big|^2$$ For $M\geq 0$, we shall note $\mathcal{W}^s(M)$ the Sobolev ball of radius $M$. Finally, we remind that, as $\|\boldsymbol{\phi}_\ell\|_{\infty}\leq 1$ for all $\ell \geq 0$, we have $s>1/2\Rightarrow \mathcal{W}^s\subset \mathcal{C}^0(\R_+)$. From now on, we will hence suppose that there exists $s>1/2$ such that $\boldsymbol{f}\in \mathcal{W}^s$.

\subsection{Banded Toeplitz matrices}
\label{Toeplitz}
Before entering into details about the kernel features, we introduce basic material on Toeplitz matrices. Most of it is inspired by \citet{BGSiam} and \citet{CPR}.
\\Let $a=(a_\ell)\in \ell^1(\mathbb{Z})$ be a sequence of real numbers. We remind from section \ref{Galerkin section} that we note $T(a)$ the infinite Toeplitz matrix defined by
$$T(a)=\begin{pmatrix}
a_0&a_{-1}&a_{-2}&\hdots&\hdots
\\a_{1}&a_0&a_{-1}&\hdots&\hdots
\\a_2&a_1&a_0&\hdots&\hdots
\\\vdots&\ddots&\ddots&\ddots&\hdots
\end{pmatrix}$$
and  $T_\ell(a)\in M_\ell(\R)$ the truncated Toeplitz matrix defined as $$(T_\ell(a))_{i,j}=(T(a))_{i,j},\;i,j\leq \ell+1$$
The Toeplitz matrices $T(a)$ and $T_\ell(a)$ are naturally linked to the two respective Laurent series
$$a(z)=\Sum_{k=-\infty}^{\infty}a_k z^k \text{ and }a_\ell(z)=\Sum_{k=-\ell}^{\ell}a_k z^k$$
We will indifferently refer to the vector $a$ or the corresponding Laurent serie. The spectral norm of $T(a)$ is related to the behaviour of $a(z)$, as illustrated in the following proposition.
\begin{proposition}
\label{Toeplitz norm}
Let $a\in \ell^1(\mathbb{Z})$. Let $\mathcal{C}$ stand for the complex unit circle. We have 
$$\|T(a)\|_{\text{op}} = \|a(z)\|_{circ}$$
where $\|a(z)\|_{circ}\overset{\Delta}{=}\displaystyle{\sup_{z\in \mathcal{C}}} \big|\Sum_{\ell=-\infty}^{\infty} a_\ell z^\ell \big|$.
A simple corollary is the following inequality
$$\|T(a)\|_{\text{op}} \leq \Sum_{\ell=-\infty}^{\infty} |a_\ell| $$
\end{proposition}
In particular, Proposition \ref{Toeplitz norm} applies to the case of truncated Toeplitz matrices $T_\ell(a)$. Moreover, if $a$ has no zero on the complex unit circle, we have 
\begin{align} \label{Toeplitz norm limit} \displaystyle{\limsup_{\ell\to \infty}} \|T_\ell(a)\|_{\text{op}} <\infty \text{ and } \displaystyle{\lim_{\ell\to \infty}} \|T_\ell(a)\|_{\text{op}}=\|T(a)\|_{\text{op}} \end{align}
Now suppose that $a$ and $a'$ both generate lower triangular Toeplitz matrices (i.e. $a_k=a_k'=0$ if $k<0$). Then the following equalities hold for all $\ell \geq 0$:
\begin{align}
\label{Toeplitz product and inverse}
T_\ell(a)T_\ell(a')=T_\ell(a')T_\ell(a)=T_\ell(aa')\text{ and }T_\ell(a)^{-1}=T_\ell(1/a)
\end{align}
In other words, the matrix multiplication (resp. inversion) is equivalent to a power serie multiplication (resp. inversion). 
\subsection{Degree of ill posedness}

We now need to precise the properties of $\boldsymbol{K}$ as a blurring operator of $L^2(\R_+)$. Usually the operator $\boldsymbol{K}$ is not compact, and the problem \eqref{Laplace deconvolution det} is ill-posed. This results in practical unstabilities when trying to invert equation \eqref{Discretized model} from discrete observations. The quantification of the ill-posedness of the problem is specified by the introduction of a constant, called degree of ill-posedness (DIP) of the problem (see \citet{NP}, \citet{MP} for a generic review). We adapt this concept to our framework, and make the following assumption.
\begin{assumption}[Degree of ill-posedness of $\boldsymbol{g}$]
\label{DIP}
There exists $\nu \geq 0$, $Q\geq 0$ such that, for all $\ell \geq 0$,
$$ \|(\boldsymbol{K}^\ell)^{-1}\|_{\text{op}} \leq Q (\ell\vee 1)^\nu$$
$\nu$ is called degree of ill-posedness of $\boldsymbol{g}$ (or equivalently of $\boldsymbol{K}$). We note $\mathcal{K}_\nu(Q)$ the set of functions which satisfy this assumption.
\end{assumption}
We shall see examples of kernels satisfying this assumption further. For the moment, we concentrate on the treatment of observations \eqref{Discretized model with error} in the context we just described.
\subsection{Algorithms and rates of convergence}
\label{AlgI}
The main challenge which remains to be treated now is to articulate the two critical steps of inversion and regularization, via adapted procedures. For example, let us give a brief overview of the methodology in \citet{CPR}: Let $\ell\in \mathbb{N}$, and let $\Lambda$ be the following contrast function, defined on $\mathbb{R}^\ell$ by
$$\Lambda:t\mapsto\|t\|^2-2\langle t,(\boldsymbol{K}^\ell)^{-1}\boldsymbol{y}^\ell \rangle$$
Note $\|.\|_{\text{op}}$ the spectral norm and $\|.\|_{\text{HS}}$ the Hilbert Schmidt norm. A model selection is performed on the maximal level $L$, by introducing the following penalizing factor ($B>0$ is an arbitrary constant):
$$\text{pen}(\ell)= 4\sigma^2 T_n n^{-1}\Big((1+B)\|\sqrt{\boldsymbol{Q}^\ell}\|_{\text{HS}}^2 +(1+B)^{-1}(\nu+1)\|\sqrt{\boldsymbol{Q}^\ell}\|_{\text{op}}^2 \log \ell\Big) $$
where $\boldsymbol{Q}^\ell=(\boldsymbol{K}^\ell)^{-1}\boldsymbol{\Omega}_\ell {^{\boldsymbol{t}}}(\boldsymbol{K}^\ell)^{-1}$ and $\sqrt{\boldsymbol{Q}^\ell}$ is a lower triangular matrix satisfying $\sqrt{\boldsymbol{Q}^\ell}\,{^{\boldsymbol{t}}} \sqrt{\boldsymbol{Q}^\ell}=\boldsymbol{Q}^\ell$.
The maximal level $\tilde{L}$ is hence chosen as 
$$\tilde{L}=\underset{\ell\leq \ell(n)}{\text{argmin}} \{ \Lambda^2\big((\boldsymbol{K}^\ell)^{-1} \boldsymbol{y}^\ell\big) +\text{pen}(\ell)\} $$
where $\ell(n)$ is a large enough resolution level, possibly depending on $n$, and the ensuing estimator of $\boldsymbol{f}$ is 
$$(\boldsymbol{K}^{\tilde{L}})^{-1} \boldsymbol{y}^{\tilde{L}} $$
We follow here a different path: we suppose that the target function belongs to a Sobolev-Laguerre space, and perform thresholding techniques in a minimax framework. Furthermore, our results are asymptotic with regard to $\ep,\de$. Would $\boldsymbol{g}$ be known, the estimation of $\boldsymbol{f}$ from observations \eqref{Discretized model with error} amounts to solving a standard inverse problem with signal noise. To this end, a prolific litterature is at disposal (a selected list is \citet{Donoho}, \citet{AS}, \citet{CHR}). In order to take into account the presence of noise in the operator, we shall hence apply a preliminar regularizing thresholding procedure to the noisy operator $\boldsymbol{K}_\de$ in order to ensure the stability of the further inversion step. To that end, define the maximal level as 
\begin{align}
\label{max level}
L^{\textbf{I}}=\lambda\Big(\ep\sqrt{|\log \ep|} \vee \de|\log \de|\Big)^{\frac{-1}{\nu+1}}
\end{align}
with $\lambda$ a positive constant. Define also the two thresholding levels
{\footnotesize
\begin{align}
\label{Operator level}
O_{\ell,\de}&=\kappa \big((\ell\vee 1) \log (\ell\vee 2)\big)^{1/2}\de\sqrt{|\log \de|}
\\ \label{Signal level} S_{\ell,n}^{\textbf{I}}&=
(\ell\vee 1)^\nu \Big(\tau_{sig}\ep\sqrt{|\log \ep|}\vee \tau_{\text{op}}\de|\log \de|\Big)
\end{align}
}
For $\ell \geq 0$, note $\boldsymbol{\zeta}_\ell=\langle(\boldsymbol{K}_\de^\ell)^{-1}\ind{\|(\boldsymbol{K}_\de^\ell)^{-1}\|_{\text{op}}< O_{\de,l}^{-1}}\boldsymbol{y}^\ell,\boldsymbol{\phi}_\ell\rangle$. The estimator $\widetilde{\boldsymbol{f}}^{\textbf{I}}$ of $\boldsymbol{f}$ is defined by 
\begin{align*}
\widetilde{\boldsymbol{f}}^{\textbf{I}}=\Sum_{\ell\leq L^{\textbf{I}}} \boldsymbol{\zeta}_\ell \ind{|\boldsymbol{\zeta}_\ell|>S_\ell}\boldsymbol{\phi}_\ell
\end{align*}
We call this procedure $\AlgI$. The preliminary threshold performed on $(\boldsymbol{K}_\de^\ell)^{-1}$ ensures its proximity with $(\boldsymbol{K}^\ell)^{-1}$ with high probability (see Lemma \ref{Neumann}). We now study the squared loss performance of the procedure.
\begin{theorem}
\label{Upper bound of Algorithm 1}
Let $M\geq 0$, $s>1/2$. Let $\nu\geq 0$, $Q\geq 0$. Suppose that Assumption \ref{Design regularity} holds for $L=L^{\textbf{I}}$. Then for sufficiently large thresholding constants $\kappa,\tau_{sig}$ and $\tau_{\text{op}}$,
\begin{align*}
\sup_{\boldsymbol{f}\in \mathcal{W}^s(M) \atop \boldsymbol{g}\in \mathcal{K}_\nu(Q)} \E\|\widetilde{\boldsymbol{f}}^{\textbf{I}}-\boldsymbol{f}\| \lesssim \Big(\de|\log \de|\Big)^{\frac{2s}{2(s+\nu)+1}}\vee\Big(\ep\sqrt{|\log \ep|}\Big)^{\frac{2s}{2(s+\nu)+1}}
\end{align*}
where $\lesssim$ means inequality up to a constant depending only on $\lambda,\kappa,\tau_{sig},\tau_{\text{op}},s,M,\nu,Q$.
\end{theorem}
The rates in Theorem \ref{Upper bound of Algorithm 1} reveal two components, accounting respectively for the imprecision in the observation of the operator and the signal. The latter is fairly classic in non parametric statistics (\citet{NP}, \citet{JKPR}) where it is also optimal, while the former is standard (and optimal too) in blind deconvolution on Hilbert spaces (\citet{EK}, \citet{HR}). Thus, we do not study the optimality of these rates in this paper, but rather concentrate on a more specific framework related to the problem of interest.
\section{Adaptation to the standard framework of Laplace deconvolution}

We now discuss the adapation of our algorithm in the mainstream framework of Laplace deconvolution, as exposed in \citet{APR} or \citet{CPR}. As we shall see, this more restrictive framework allows to treat observations \eqref{Discretized model with error} more efficiently. To this end, we first define a more restrictive version of the degree of ill-posedness.

\begin{assumption}[Second kind degree of ill-posedness]
\label{DIP 2}
Note $\boldsymbol{\gamma}_k=\langle (1/\boldsymbol{\dot{g}}),\boldsymbol{\phi}_k \rangle$, so that $(1/\boldsymbol{\dot{g}})(z)=\Sum_{k\geq 0} \boldsymbol{\gamma}_k z^k $.
There exists $\nu > 0$, there exists $Q_2, Q_1>0$, such that for all $\ell \geq 0$,
\begin{align}
\label{DIP 2.1} \Sum_{k=0}^\ell \boldsymbol{\gamma}_k^2 &\leq Q_2 (\ell\vee 1)^{2\nu-1}
\\ \label{DIP 2.2} \Sum_{k=0}^\ell \Sum_{n=0}^k \boldsymbol{\gamma}_n^2 &\geq Q_1 (\ell\vee 1)^{2\nu}
\end{align}
\end{assumption}
For $Q=(Q_1,Q_2)$, we note $\mathcal{G}_{\nu}(Q)$ the set of functions $\boldsymbol{g}\in L^2(\R_+)$ such that Assumption \ref{DIP 2} holds. Note that the validity of this assumption automatically entails $Q_1\leq \big(1+\frac{2^{2\nu}}{2\nu}\big)Q_2$. Note also that the left term in \eqref{DIP 2.2} is the Hilbert-Schmidt norm of $(\boldsymbol{K}^\ell)^{-1}$. Thus, Assumption \ref{DIP 2} is more restrictive that Assumption \ref{DIP}. However, it is satisfied by a natural class of functions $\boldsymbol{g}$:
\begin{proposition}[\citet{CPR}, Lemma 3/ Lemma 5]
\label{DIP 2 Proposition}
Suppose that there exists $C,\nu>1/2$, $\mu\in \mathcal{C}$ and $w(z)=\prod_{i=1}^N (z-\mu_i),\,|\mu_i|>1$ a polynomial function with no pole inside of the complex unit disc, such that \begin{align}\label{Kernel decomposition} \boldsymbol{\dot{g}}(z)=C w(z)(\mu-z)^{\nu}
\end{align}
Then Assumption \ref{DIP 2} is satisfied.
Furthermore, if $w\equiv 1$ and $\nu\geq 0$, then $|\boldsymbol {\gamma}_\ell| \sim \frac{\ell^{\nu-1}}{\Gamma(\nu)}$.
\end{proposition}
For completeness, we give a proof of Proposition \ref{DIP 2 Proposition} in section \ref{Proofs}. We now turn to the standard framework of Laplace deconvolution, as exposed in \citet{APR} and \citet{CPR}. To this end, we define the following assumptions  concerning the kernel $\boldsymbol{g}$.
\begin{assumption}
\label{assumptions}
\hspace{1cm}
\begin{enumerate}[label=(A\arabic*)]
\item \label{A1} There exists an integer $r\geq 1$ such that $$\frac{d^j\boldsymbol{g}}{dt^j}\big|_{t=0}=\begin{cases}0&\text{ if } j=0,1,...,r-2\\B_r \neq 0 &\text{ if } j=r-1 \end{cases}$$
\item \label{A2} $\boldsymbol{g}\in L^1\big([0,+\infty)\big)$ is $r$ times differentiable and $\boldsymbol{g}^{(r)}\in  L^1\big([0,+\infty)\big)$.
\item \label{A3} The Laplace transform of $\boldsymbol{g}$ has no zeros with non negative real parts except for the zeros of the form $\infty+ib$.
\end{enumerate}
\end{assumption}
The consequences of these assumptions are well formulated in the terms of the preceding framework:
\begin{proposition}[\citet{CPR}, Lemma 3]
\label{Theorem CPR}
Suppose that Assumptions \ref{A1}, \ref{A2} and \ref{A3} hold. Then the hypotheses of Proposition \ref{DIP 2 Proposition} are satisfied with $\mu=1$, $\nu=r$.
\end{proposition}
Hence, Assumption \ref{DIP 2} is verified with $\nu=r$ and $\AlgI$ applies. However, Assumption \ref{DIP 2} provides additional information on the behaviour of $(1/\boldsymbol{\dot{g}})$. We adapt $\AlgI$ to this new framework, by operating the following changes:
\begin{itemize}
\item Set the maximal level to $$L^{\textbf{II}}=\lambda\Big(\ep\sqrt{|\log \ep|} \vee \de|\log \de|\Big)^{-1}$$
\item Set the signal thresholding level to {\footnotesize
\begin{align}
\label{Signal level II} S_{\ell,n}^{\textbf{II}}&=\begin{cases}
\| (\boldsymbol{K}_\de^\ell)^{-1}\|_{\text{HS}} (\ell\vee 1)^{-1/2} \Big( \tau_{sig}\ep\sqrt{|\log \ep|} \vee \tau_{\text{op}}\de |\log \de|\Big) &\text{ if }\|(\boldsymbol{K}_\de^\ell)^{-1}\|_{\text{op}}< O_{\ell,\de}^{-1}
\\ +\infty &\text{ if } \|(\boldsymbol{K}_\de^\ell)^{-1}\|_{\text{op}}\geq  O_{\ell,\de}^{-1}
\end{cases}
\end{align}
}
\end{itemize}
where $\|A\|_{\text{HS}}=\sqrt{\text{Tr}(^tA A)}$ is the Hilbert-Schmidt norm. We call the modified procedure $\AlgII$ and note $\widetilde{\boldsymbol{f}}^{\textbf{II}}$ the corresponding estimated function. A notable gain of this new algorithm is its independence with regard to the parameter $\nu$. Indeed, Assumption \ref{DIP 2} allows us to use $\|(\boldsymbol{K}_\de^\ell)^{-1}\|_{\text{HS}}$ in \eqref{Signal level} as a substitute of $\ell^\nu$, and to overesimate the 'true' maximal level $L^{\textbf{I}}$. Its performances are exposed in the next theorem:
\begin{theorem}
\label{Upper bound of Algorithm 2}
Let $M\geq 0$. Let $\nu>0$, $Q_2,Q_1>0$ and $s>1/2$. Suppose that Assumption \ref{Design regularity} holds with $L=L^{\textbf{II}}$. Then for sufficiently large thresholding constants $\kappa,\tau_{sig}$ and $\tau_{\text{op}}$,
\begin{align*}
\sup_{\boldsymbol{f}\in \mathcal{W}^s(M) \atop \boldsymbol{K}\in \mathcal{G}_\nu(Q)} \E\|\widetilde{\boldsymbol{f}}^{\textbf{II}}-\boldsymbol{f}\| \lesssim \Big(\de|\log \de|\Big)^{\frac{s}{s+\nu}}\vee\Big(\ep\sqrt{|\log \ep|}\Big)^{\frac{s}{s+\nu}}
\end{align*}
where $\lesssim$ means inequality up to a constant depending only on $\lambda,\kappa,\tau_{sig},\tau_{\text{op}},s,M,\nu,Q_1,Q_2$.
\end{theorem}
Thus, in addition to the adaptivity over the parameter $\nu$, the strengthening of Assumption \ref{DIP} via \eqref{DIP 2.1} and \eqref{DIP 2.2} allows to improve on the rates of Theorem \ref{Upper bound of Algorithm 1} with regard both to the operator and signal noise. Our next result shows that the rate achieved in Theorem \ref{Upper bound of Algorithm 2} is indeed optimal, up to logarithmic terms. The lower bound will not decrease for increasing noise levels $\de$ and $\ep$, whence it suffices to provide separately the cases $\de=0$ and $\ep=0$.
\begin{theorem}
\label{Lower bound}
Let $s>1/2$, let $M\geq 0$ $\nu>1/2$ and $Q_2\geq c_\nu Q_1>0$. Here $c_\nu$ is a constant depending only on $\nu$ which will will not seek to precise. We have
$$\inf_{\tilde{f}}\sup_{f\in \mathcal{W}^s(M) \atop \boldsymbol{g}\in\mathcal{G}_\nu(Q)} \E\|\tilde{\boldsymbol{f}}-\boldsymbol{f}\| \gtrsim \de^{\frac{s}{s+\nu}}|\log \de|^{-1}\vee\ep^{\frac{s}{s+\nu}}|\log \ep|^{-1}$$
where the infimum is taken among all estimators $\tilde{f}$ of $f$ based on observations \eqref{Discretized model with error}.
\end{theorem}
Combining Theorem \ref{Upper bound of Algorithm 1} together with Theorem \ref{Lower bound}, we conclude that our algorithm is minimax over $\mathcal{W}^s(M)$ to within logarithmic terms in
$\ep$ and $\de$, uniformly with regard to the blurring kernel $\boldsymbol{g}\in\mathcal{G}_\nu(Q)$.

\section{Practical performances}

In this section we study the practical performances of the two procedures developped above. Note that three potential sources of errors may contaminate the quality of the observations in \eqref{Discretized model with error} : the signal precision $\sigma\sqrt{\frac{T_n}{n}}$, the operator precision $\de$ and the design quality $\|\boldsymbol{\Omega_{\ell}}\|_{\text{op}}$. We shall hence emphasize their influence in the estimation of $\boldsymbol{f}$, as well as their respective interactions.
\\Our first aim is to study the interaction between the effect of signal and kernel noise in the two procedures of reconstruction. To this end, we will isolate them from the effect of the design, and suppose that the latter is ideally conditionned by setting $\boldsymbol{\Omega}_\ell=\boldsymbol{I}_\ell$. Let us start by a few precisions concerning the tuning parameters of \textbf{Algorithm I} and \textbf{II}. The setting up of these procedures requires the preliminary definition of $\lambda,\kappa,\tau_{sig}$ and $\tau_{\text{op}}$.
 \\\textbf{Tuning parameters}: for the definition of the maximal level of resolution, we set $\lambda=1$ for both algorithms. The concrete choice of adequate thresholding constants $\kappa$ and $\tau$ is a complex issue. Our practical choices will be based on the following remark, inspired by \citet{Donoho1994}: in the case of direct estimation on real line, the universal threshold which is both efficient and simple to implement, takes the form $2\sqrt{|\log \ep|}$. A consistent interpretation is to consider that this threshold should kill any pure noise signal. We will adapt this reasoning to the case of interest.
\\\underline{Choice of $\kappa$} : we use as a benchmark the case where $\boldsymbol{g}\equiv 0$. Given $\de$ large enough, we define $\kappa$ as the smallest value $\kappa_\de$ such that , for all $\ell\leq 10$, $\ind{\|(\boldsymbol{K}_\de^\ell)^{-1}\|_{\text{op}}< O_{\de,l}^{-1}}=0$. The results are reported in Table \ref{Choice of kappa} and give $\kappa=0.3$.
\\\underline{Choice of $\tau_{sig}$ and $\tau_{op}$}: It is clear that the role of  $\tau_{sig}$ and $\tau_{op}$ is to control the influence of the signal (resp.  the operator) error. To choose $\tau_{sig}$ (resp. $\tau_{op}$), we therefore set $\ep_{sig}>\de_{sig}>0$ (resp. $\de_{op}>\ep_{op}>0$) large enough. We resort to the case $\boldsymbol{f}\equiv 0$ as a benchmark: we have $\langle \boldsymbol{f},\boldsymbol{\phi}_\ell \rangle=0$ for $\ell\geq 1$, consequently the observations $\langle \boldsymbol{g}_{\ep_{sig}}, \boldsymbol{\phi}_\ell \rangle$, $\ell\geq0$ are pure noise. We hence simulate $\boldsymbol{K}_{\de_{sig}}$ and, integrating the precedently computed value of $\kappa$, apply the procedure for increasing values of $\tau_{sig}$ (resp. $\tau_{op})$ until all the computed coefficients $\langle \tilde{\boldsymbol{f}}^i,\boldsymbol{\phi}_\ell\rangle$ ($i=$\textbf{I},\textbf{II}) are killed for $\ell\leq 10$. The results are reported in Table \ref{Choice of tau}.

\begin{table}
\centering
\begin{tabular}{c|ccc}
$\kappa$ &0.1& 0.2&0.3
\\ \hline
$N$ &3&1&0
\end{tabular}
\caption{\footnotesize Choosing of $\kappa$. $N$ is the average number, computed on a basis of $10$ realizations, of levels $\ell\leq 10$ such that $\|(\boldsymbol{K}_\de^\ell)^{-1}\|_{\text{op}}< O_{\de,l}^{-1}(\kappa)$. We have $\de=10^{-2}$.}
\label{Choice of kappa}
\end{table}

\begin{table}
\centering
\begin{tabular}{c|cccccccc}
$\tau_{sig}$&0.3&0.4&0.5&0.6&0.7&0.8&0.9&1
 \\ \hline $N_{\textbf{I}}$&1&1&0&0&0&0&0&0
 \\ \hline $N_{\textbf{II}}$&7&5&4&3&2&1&1&0
\end{tabular}
\hspace{1cm}\begin{tabular}{c|c}
$\tau_{\text{op}}$&0.1
 \\ \hline $N_{\textbf{I}}$&0
  \\ \hline $N_{\textbf{II}}$&0
\end{tabular}
\caption{\footnotesize Choosing of $\tau$. For $(\de_{sig},\ep_{sig})=(\ep_{\text{op}},\de_{\text{op}})=(10^{-2},10^{-1})$ and each value of $\tau$, we computed $10$ times the described procedure and reported $N_i$ the average number of remaining Laguerre coefficients for \textbf{Algorithm} $i$.}
\label{Choice of tau}
\end{table}
\begin{figure}
\begin{center}
\subfigure[\textbf{Target function $\boldsymbol{f}_1$}]
    {\includegraphics[width=4.6cm,height=7cm]{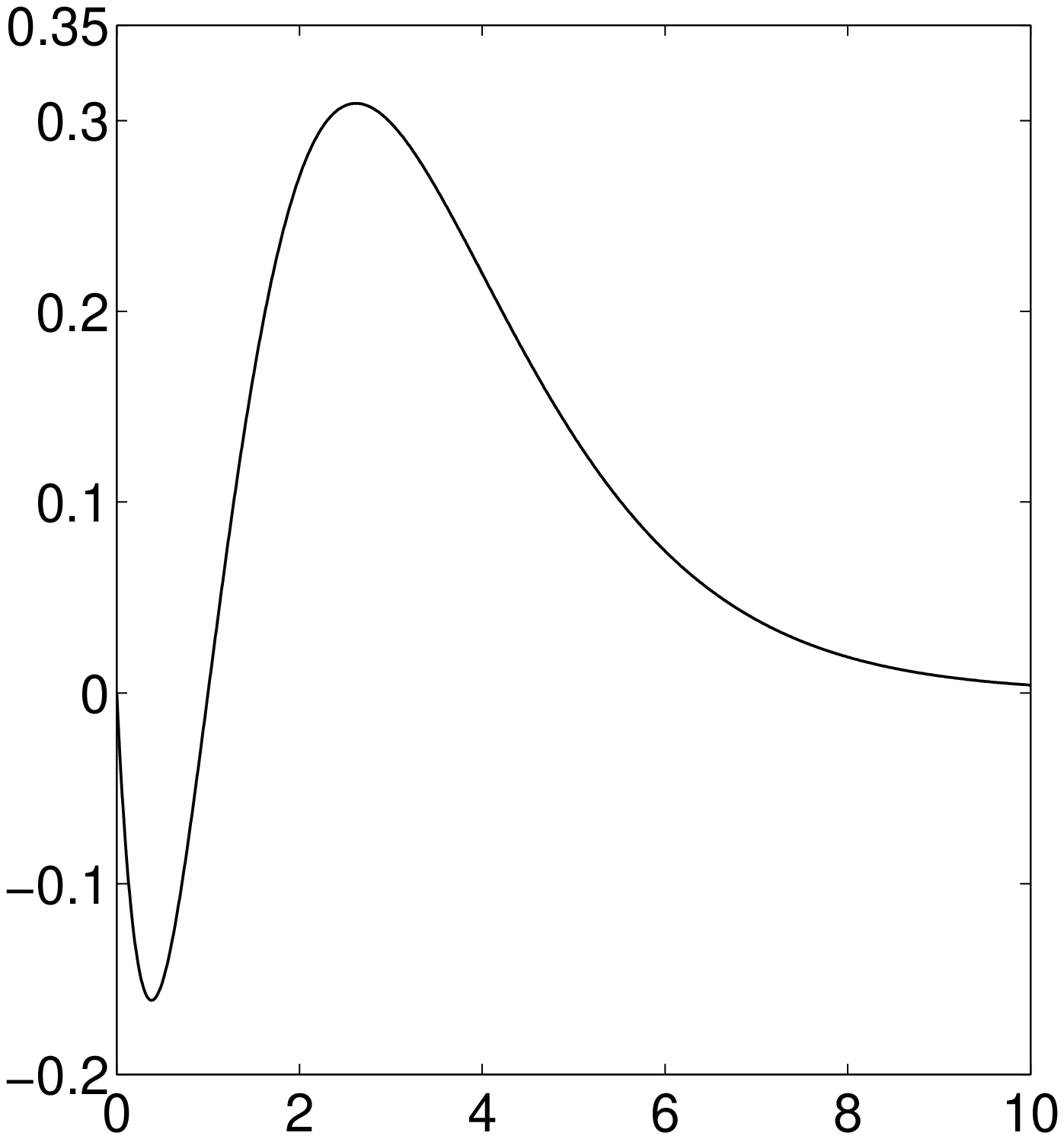}
    }    
\subfigure[\textbf{Kernel $\boldsymbol{g}$}]
    {\includegraphics[width=4.6cm,height=7cm]{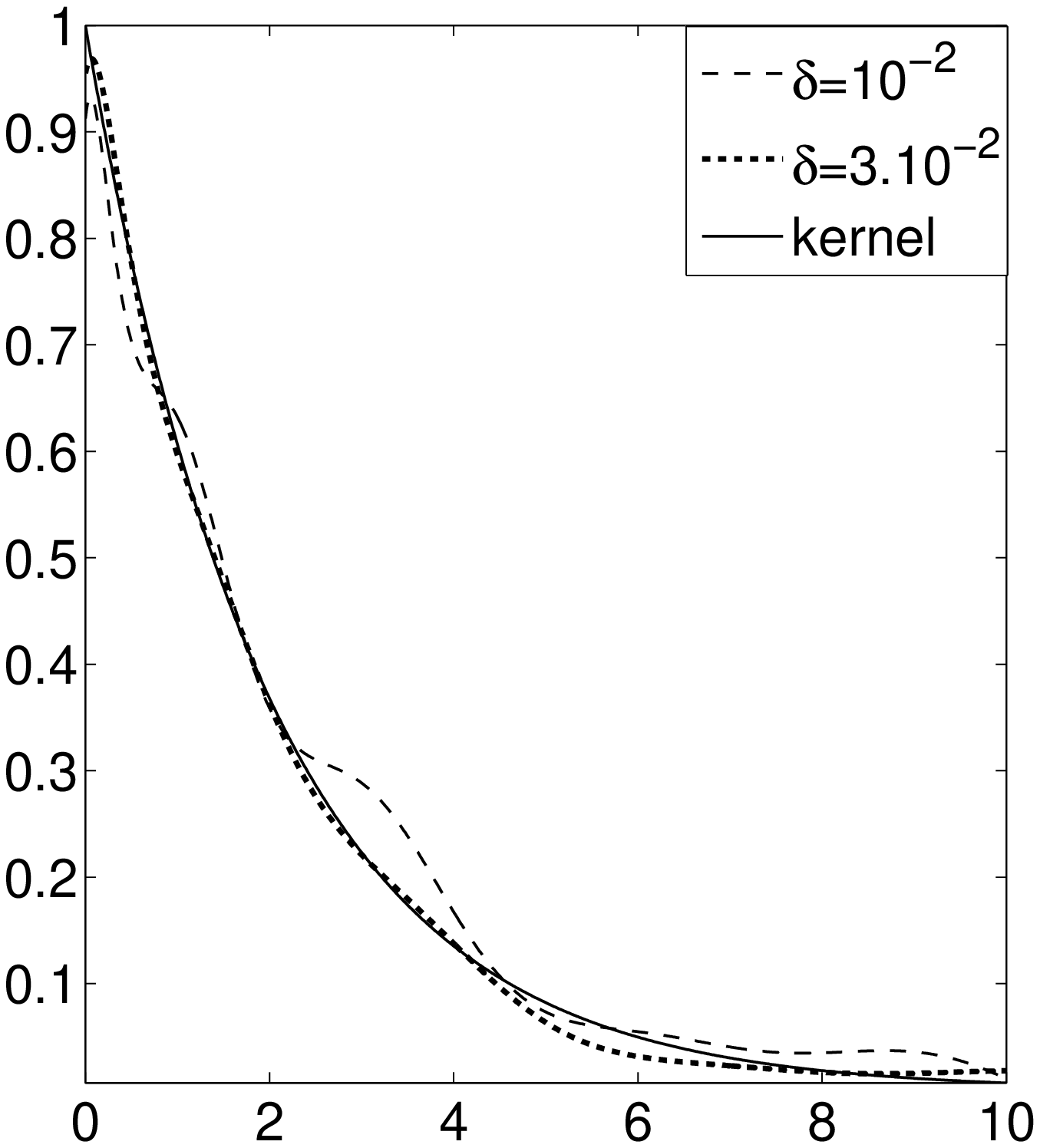}
    }
\subfigure[\textbf{$\boldsymbol{q=Kf}$}]
    {\includegraphics[width=4.6cm,height=7cm]{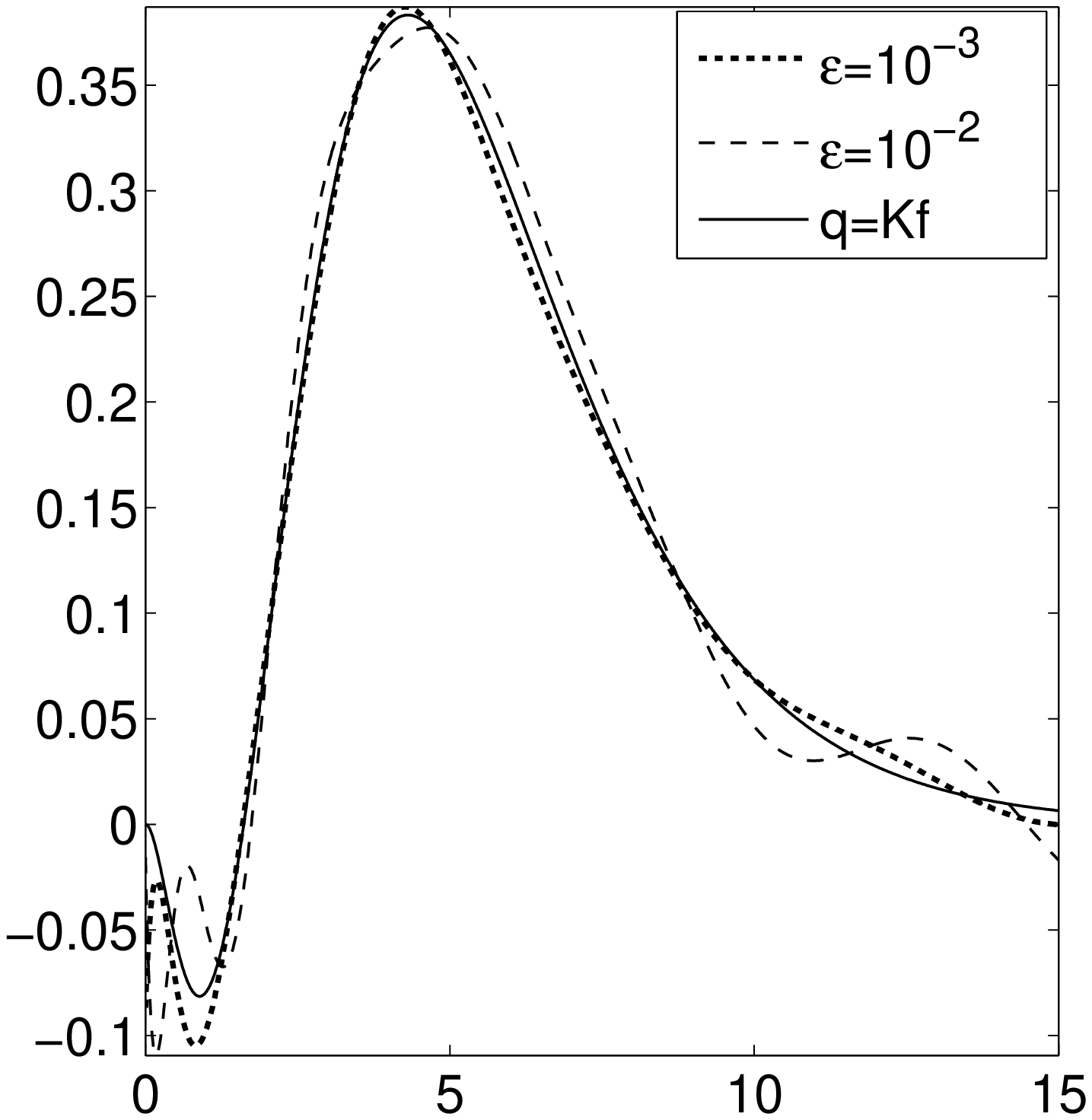}
    }
\caption{\footnotesize Datas and noisy observations of $\boldsymbol{g}$ and $\boldsymbol{q}$}
\label{Datas}
\end{center}
\end{figure}
We now apply the two procedures to the case where $\boldsymbol{f}_1(t)=(t^2-t)\exp(-t)$ and $\boldsymbol{g}=\boldsymbol{\phi}_0$ (a graphical representation of these two functions is presented in Figure \ref{Datas}). 

We have 
$$\big(1/\boldsymbol{\dot{g}}\big)(z) = (1-z)^{-1}= \Sum_{\ell \geq 0}z^\ell$$
hence Assumptions \ref{DIP} and \ref{DIP 2} are both satisfied taking $\nu=1$. For several values of $\ep$ and $\de$, we report the corresponding squared loss, computed on a basis of $500$ realisations with the use of Parseval's identity, in Table \ref{Squared loss}. The corresponding results are presented in Figure \ref{RegularPlots} for one particular realization of $\boldsymbol{\xi},\boldsymbol{b}$. The results indicate that the transition on the two types of errors occur when $\de$ is higher than $\ep$, translating a prevailing effect of the signal noise $\ep$ over the operator error $\de$ in practice. As Theorems \ref{Upper bound of Algorithm 1} and \ref{Upper bound of Algorithm 2} suggest, the second Algorithm overperforms the first in (almost) every case.
\renewcommand{\arraystretch}{1.2}
\begin{table}
\centering
\begin{tabular}{c|c|c|c|c||c|c|c|c|}
\cline{2-9}
&\multicolumn{4}{c||}{\textbf{Algorithm I}}&\multicolumn{4}{c|}{\textbf{Algorithm II}}
\\\hline
\multicolumn{1}{|c|}{\backslashbox{$\de$}{$\ep$}}&0&$10^{-3}$&$10^{-2}$&$3.10^{-2}$&0&$10^{-3}$&$10^{-2}$&$3.10^{-2}$
 \\ \hline \hline \multicolumn{1}{|c|}{$0$}&0&0.020&0.141&0.348&0&0.012&0.109&0.312
 \\\hline \multicolumn{1}{|c|}{$10^{-3}$}&0.004&0.020&0.141&0.352&0.005&0.012&0.108&0.301
 \\\hline \multicolumn{1}{|c|}{$10^{-2}$}&0.047&0.054&0.143&0.344&0.053&0.039&0.116&0.318
 \\\hline \multicolumn{1}{|c|}{$3.10^{-2}$}&0.170&0.169&0.190&0.348&0.118&0.109&0.145&0.324
\\\hline
\end{tabular}
\caption{\footnotesize Normalized mean squared error of the two procedures applied to the functions $\boldsymbol{f}_1$ and $\boldsymbol{g}$. The computations were performed using a monte carlo method on $500$ realizations.}
\label{Squared loss}
\end{table}

\begin{figure}
\centering
\begin{center}
\subfigure[\textbf{Algorithm I}]
    {\includegraphics[width=7cm,height=7cm]{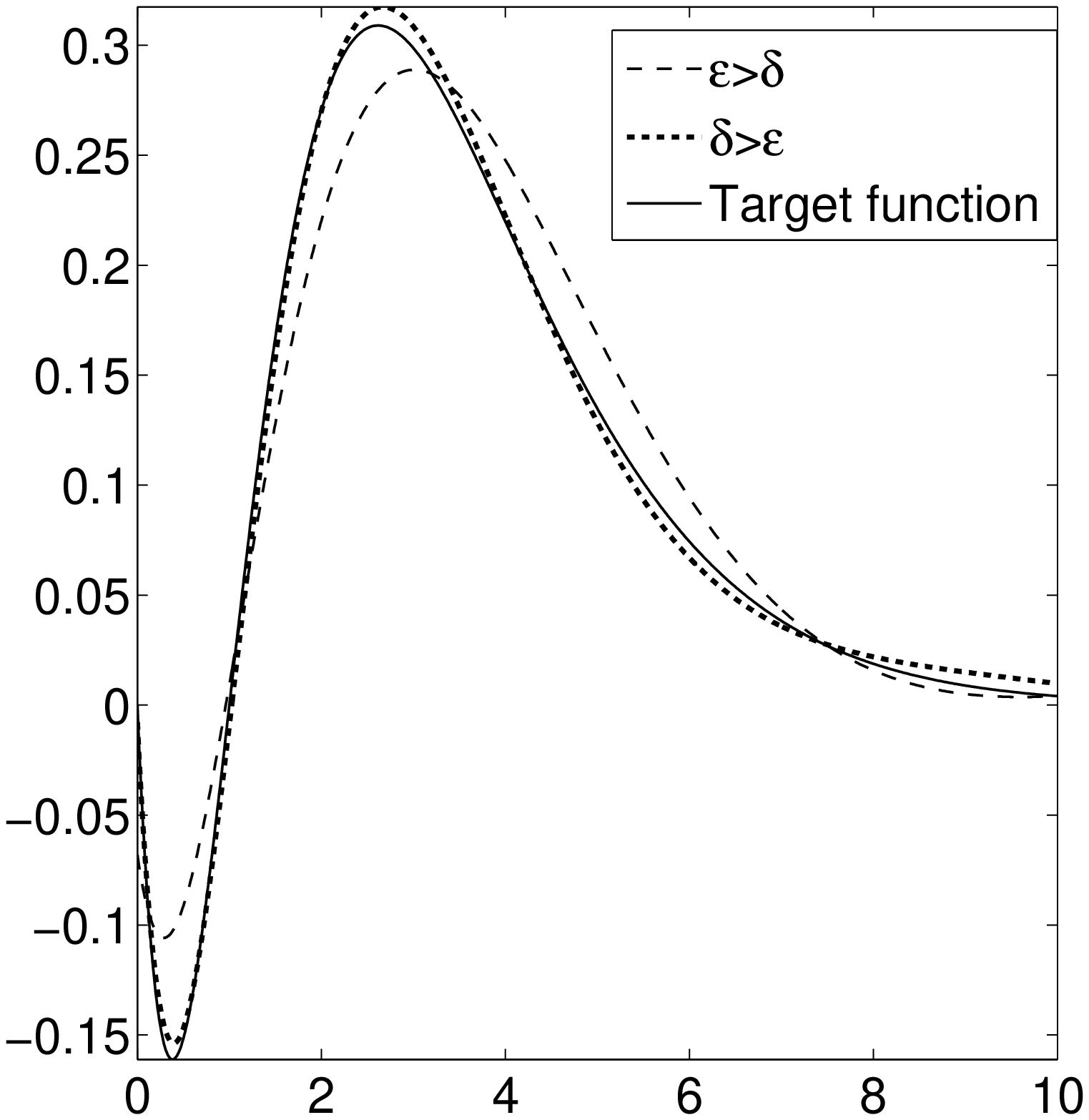}
    }    
\subfigure[\textbf{Algorithm II}]
    {\includegraphics[width=7cm,height=7cm]{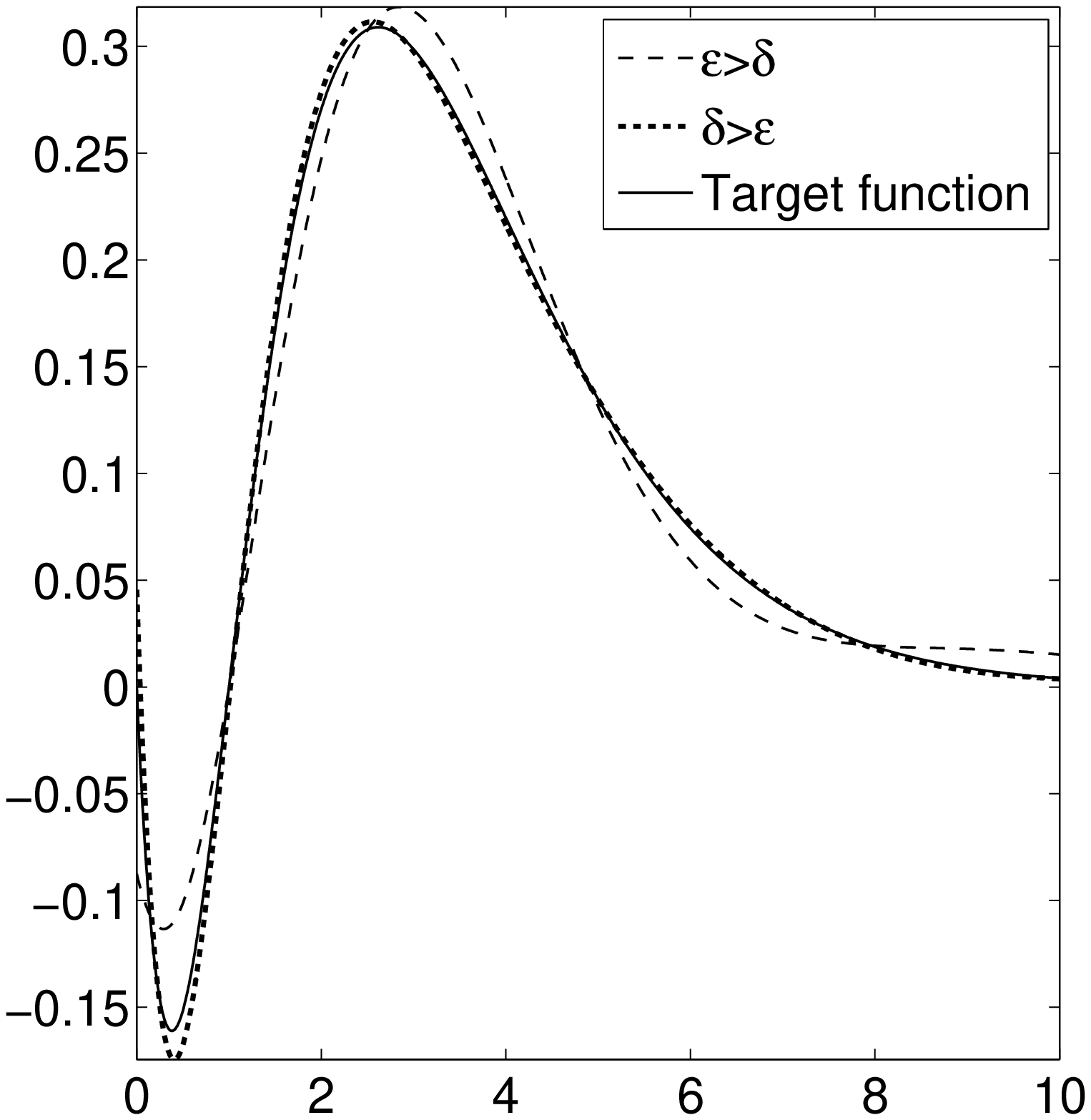}
    }
\caption{\footnotesize Estimation of $\boldsymbol{f}_1$ for predominant signal noise $(\ep,\de)=(10^{-2},10^{-3})$ and predominant operator noise $(\de,\ep)=(10^{-3},10^{-2})$.}
\label{RegularPlots}
\end{center}
\end{figure}

\textbf{Discussion on the design irregularity}: to control the squared risk of the two procedures, one needs condition \ref{Design regularity} to be fulfilled. If not, the eigenvalues of the matrix $\boldsymbol{\Omega}_L$ become potentially too large, and observations \eqref{Discretized model} are not conveniently treatable. In this case, it is preferable to lower the maximal level down to a point where $\|\boldsymbol{\Omega}_L\|_{op}$ remains under control. To this end, we change the maximal level of the two respective procedures to 
$$N^{i}=L^i\wedge \max\{\ell\geq 0\text{ s.t. }\|\boldsymbol{\Omega}_\ell\|_{op}\leq \alpha \},\hspace{0.3cm} i=\textbf{I},\textbf{II} $$
where $\alpha$ is an arbitrary thresholding constant, set to $1.5$ in the sequel. We now fix $\sigma=\de=10^{-2}$ and chose the design points $t_i$ as $t_i=100i/n$ for $n=200,250,750$ and $1000$. Taking the same kernel $\boldsymbol{g}=\boldsymbol{\phi}_0$, and setting $\boldsymbol{f}_2(t)=(t^{1/2}-t)\exp(-t)$, we compare the performances of the new choice $N^i$ to the previous one $L^i$, by computing the respective mean squared losses on a basis of $500$ observations and report the result in Table \ref{Squared loss design}. The results show a minor effect of the design ill-posedness on \textbf{Algorithm I}, since $L^i$ is usually already smaller than $N^i$. However, the gain is notable for \textbf{Algorithm II} when $n\leq 250$. To illustrate this point, we plot in Figure \ref{Design Plots} the corresponding results when $n=200$.

\renewcommand{\arraystretch}{1.5}
\begin{table}
\centering
\begin{tabular}{c|c||c|c|c|c|c}
\cline{2-6} &\backslashbox{}{$n=$}&200&250&500&750
 \\ \hline 
\multicolumn{1}{|c|}{\multirow{3}{*}{\textbf{Algorithm I}}}&$\big(L^\textbf{I},N^\textbf{I}\big)$&(6,6)&(6,6)&(6,6)&(6,6)
\\\cline{2-6}\multicolumn{1}{|c|}{}&\textbf{MSE}, $L^\textbf{I}$&0.273&0.270&0.264&0.258
 \\ \cline{2-6} \multicolumn{1}{|c|}{}
&\textbf{MSE}, $N^\textbf{I}$&0.275&0.272&0.264&0.257
 \\ \hline \hline
 \multicolumn{1}{|c|}{\multirow{3}{*}{\textbf{Algorithm II}}}&$\big(L^\textbf{II},N^\textbf{II}\big)$&(37,12)&(37,15)&(37,27)&(37,27)
\\\cline{2-6}\multicolumn{1}{|c|}{}&\textbf{MSE}, $L^\textbf{II}$&1.336&0.559&0.289&0.253
 \\ \cline{2-6}
\multicolumn{1}{|c|}{}&\textbf{MSE}, $N^\textbf{II}$&0.294&0.291&0.284&0.256
 \\ \hline 
 \end{tabular} 
\caption{\footnotesize Normalized mean squared error of the two procedures when the design is constituted of $200$ equispaced points on the interval $[0;100]$. We compare the performances of the two maximal resolution levels $L^i$ and $N^i$ for the parameters $\sigma=\de=10^{-2}$, $\boldsymbol{g}=\boldsymbol{\phi}_0$ and $\boldsymbol{f}_2(t)=(t^{1/2}-t)\exp(-t)$.}
\label{Squared loss design}
 \end{table}

\begin{figure}
\begin{center}
\subfigure[\textbf{Algorithm I}]
    {\includegraphics[width=7cm,height=7cm]{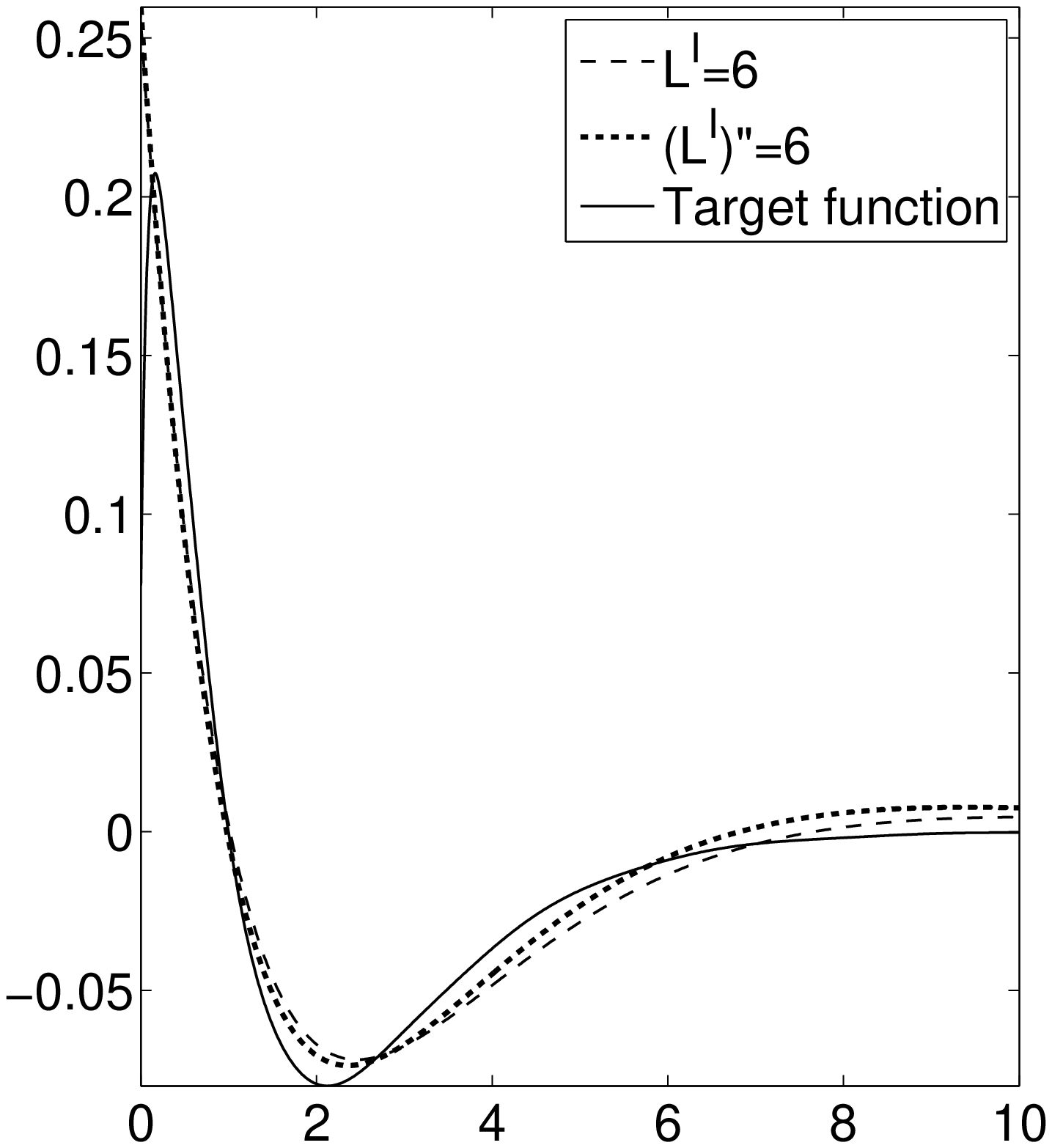}
    }    
\subfigure[\textbf{Algorithm II}]
    {\includegraphics[width=7cm,height=7cm]{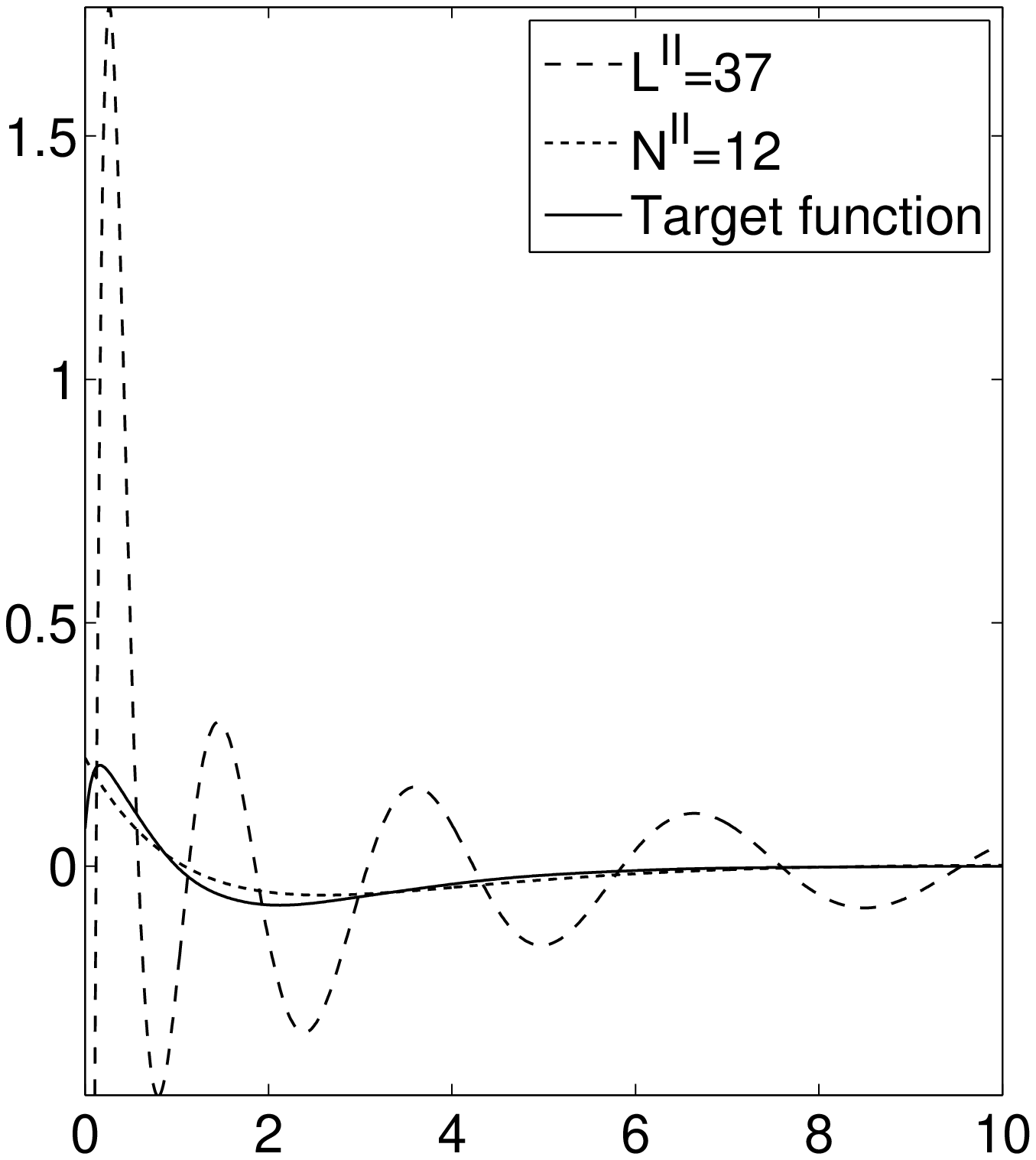}
    }
\caption{\footnotesize Result of the two different maximal levels $L^i$ and $N^i$ to estimate $\boldsymbol{f}_2$, for a particular realizaton of $\boldsymbol{b}$ and $\boldsymbol{\xi}$. The design is constituted of $200$ equidistant points of observations in $[0;100]$. The noise levels are $\sigma=\delta=10^{-2}$.}
\label{Design Plots}
\end{center}
\end{figure}

\subsection*{Back to the regression model}

We now turn back to the original model \eqref{Laplace deconvolution reg} to apply the two procedures. It is well known that this model is asymptotically equivalent to \eqref{Discretized model}, in the sense that a fine enough design will provide an estimation of the Laguerre coefficients with a negligible error when $n\to\infty$. We work with $\boldsymbol{f}_3(z)=(1-z)^{1/2}$ , $\boldsymbol{g}=\boldsymbol{\phi}_0$, $\de=10^{-2}$, and suppose that the design is constituted of the points $t_i=\Sum_{j=1}^i (\text{step}+ |X_j|)$ where $(X_j)_{j\leq n}$ is an i.i.d sequence of $\mathcal{N}(0,10^{-2})$ variables. We observe the noisy values $\boldsymbol{y}(t_i)=\boldsymbol{q}(t_i)+ \sigma \eta_i$ where $\boldsymbol{q}(z)=(1-z)^{3/2}$, and compute the Laguerre coefficients $\check{\boldsymbol{q}}_\ell$ via the approximation
$$\check{\boldsymbol{q}}_\ell \sim \Sum_{i=1}^{n-1} \frac{\boldsymbol{q}(t_i)\boldsymbol{\phi}_\ell(t_i)+\boldsymbol{q}(t_{i+1})\boldsymbol{\phi}_\ell(t_{i+1})}{2}(t_{i+1}-t_i)$$ 
We apply the two procedures and present the results on Figure \ref{Regression Plots}.

\begin{figure}
\centering
\begin{center}
\subfigure[$\text{step}=5.10^{-1};\,n=30 \atop (\text{MSE\textbf{I}},\text{MSE\textbf{II}})=(0.166,0.177)$]
    {\includegraphics[width=7cm,height=7cm]{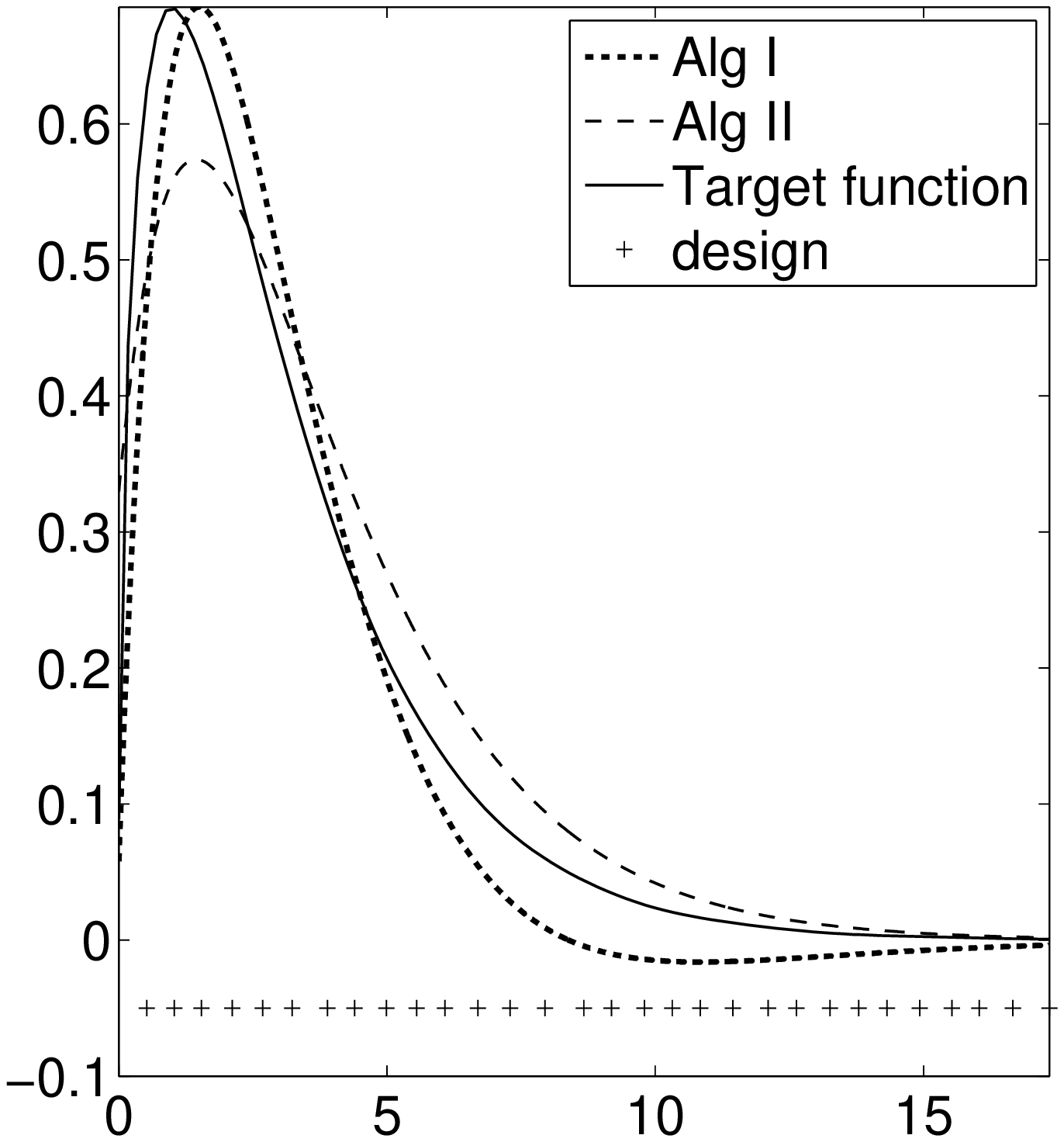}
    }    
\subfigure[{\footnotesize $\text{step}=10^{-1};\, n=100\atop (\text{MSE\textbf{I}},\text{MSE\textbf{II}})=(0.118,0.133)$}]
    {\includegraphics[width=7cm,height=7cm]{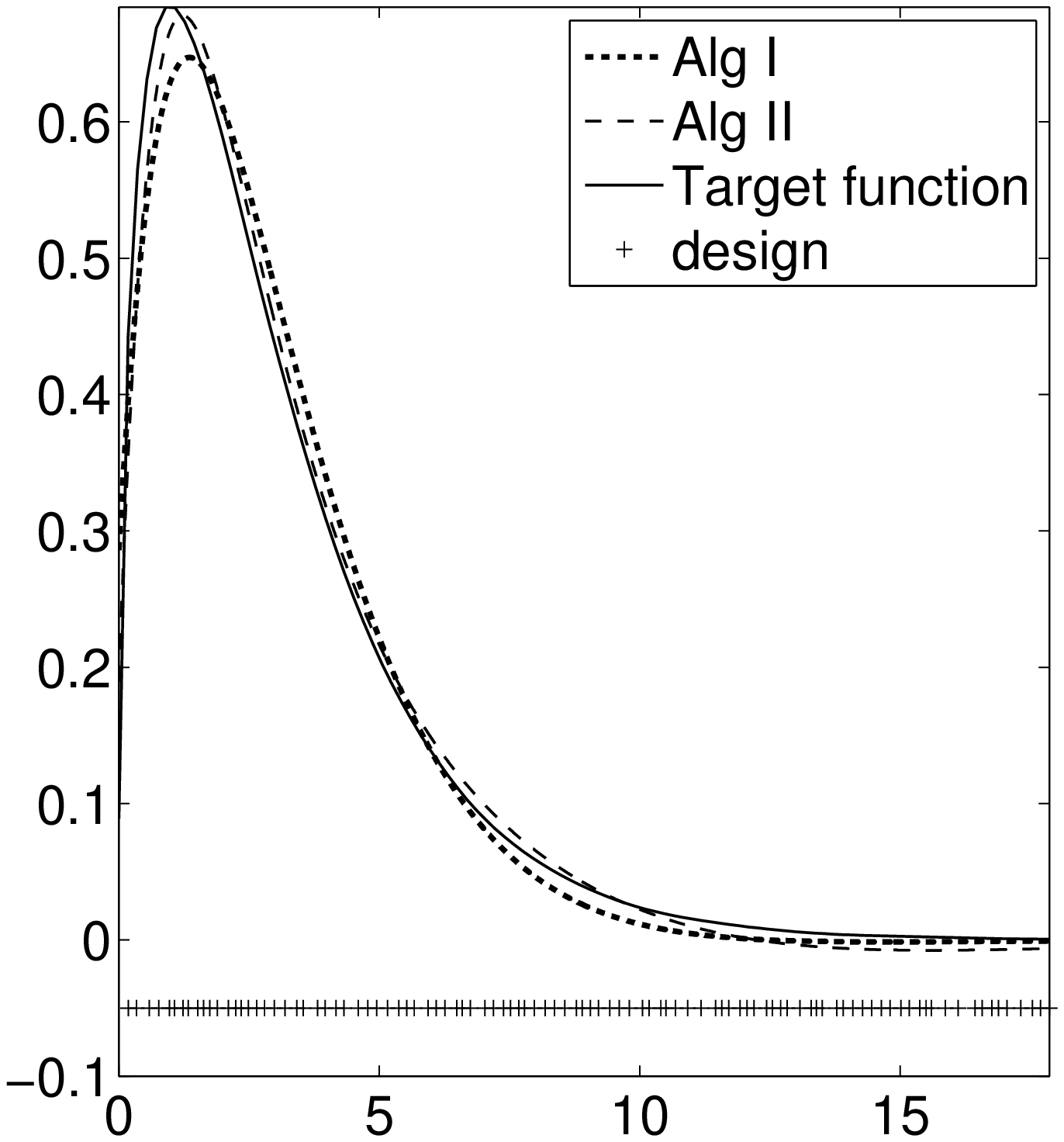}
    }
\caption{\footnotesize Adaptation of the procedures to the regression framework. Here, MSE denotes the normalized mean squared error for each algorithm (computed with 500 realizations). The target function is $\boldsymbol{f}_3$ and the noise levels are $\sigma=\de=5.10^{-2}$.}
\label{Regression Plots}
\end{center}
\end{figure}
\section{Proofs}
\label{Proofs}
In the sequel, for the sake of clarity, we suppose that $\tau_{sig}=\tau_{\text{op}}\overset{\Delta}{=}\tau$.
\subsection{Proof of Proposition \ref{DIP 2 Proposition}}
\begin{proof}
We can restrict ourselves to the case where $\mu=1$. Proposition \ref{Toeplitz product and inverse} applied to equality \eqref{Kernel decomposition} entails
\begin{align*}
\forall \ell \geq 0,\; T_\ell\big((1/\boldsymbol{\dot{g}})\big)&=C^{-1} T_\ell\big(w^{-1}\big)T_\ell\big((1-z)^{-\nu}\big)
\\T_\ell\big((1-z)^{-\nu}\big)&=C T_\ell\big(w\big)T_\ell\big((1/\boldsymbol{\dot{g}})\big)
\end{align*}
As a consequence, 
\begin{align}
\label{L2 Prop} \Sum_{k=0}^\ell \boldsymbol{\gamma}_k ^2=\|(1/\boldsymbol{\dot{g}})^\ell\|^2 &\leq C^{-1} \|T_\ell(w^{-1})\|_{\text{op}} \|(1-z)^{-\nu}_\ell\|^2 
\\ \label{HS Prop}  \text{and}\hspace{0.5cm}\|(\boldsymbol{K}^\ell)^{-1}\|_{\text{HS}} &\geq C^{-1} \|T_\ell\big(w\big)\|_{\text{op}}^{-1} \|T_\ell\big((1-z)^{-\nu}\big)\|_{\text{HS}}
\end{align}
Since $w$ is assumed to have no zeros on $\mathcal{C}$, we deduce from Proposition \ref{Toeplitz norm} that both $\|w^{-1}\|_{circ}$ and $\|w\|_{circ}$ are finite, and from \eqref{Toeplitz norm limit} that
$$\|T_\ell\big(w^{-1}\big)\|_{\text{op}}\asymp 1 \text{ and } \|T_\ell\big(w\big)\|_{\text{op}} \asymp 1$$
It remains to treat the binomial serie $(1-z)^{-\nu}$. This serie can be expanded as 
$$ (1-z)^{-\nu} = \Sum_{\ell \geq 0} (-1)^\ell \binom{-\nu}{\ell} z^\ell$$
, where $\binom{-\nu}{\ell}\overset{\Delta}{=}\frac{\Gamma(-\nu+1)}{\Gamma(\ell+1) \Gamma(-\nu-\ell+1)}$ is the generalized binomial coefficient. Furthermore, we have
 \begin{align}\label{Binom equiv}
 \binom{-\nu}{\ell}\underset{\ell\to \infty}{\sim}\frac{(-1)^\ell}{\Gamma(\nu)\ell^{-\nu+1}}
 \end{align}
which is a direct consequence of Euler's definition of the Gamma function $\Gamma(z)=\displaystyle{\lim_{k\to \infty} \frac{k!k^z}{\Pi_{i=0}^k(z+i)}} $. Since $\nu>1/2$, the serie $\Sum_k \binom{-\nu}{k}^2$ is hence divergent, and there exists $\widetilde{Q}_2,\widetilde{Q}_1>0$ such that, for all $\ell \geq 0$, 
$$\Sum_{k=0}^\ell \binom{-\nu}{k}^2 \leq \widetilde{Q}_1 (\ell\vee 1)^{2\nu-1} \hspace{0.5cm} \text{and}\hspace{0.5cm} \Sum_{k=0}^\ell \Sum_{n=0}^k \binom{-\nu}{n}^2 \geq \widetilde{Q}_2 (\ell\vee 1)^{2\nu}$$
The proof is complete thanks to \eqref{L2 Prop} and \eqref{HS Prop}.
\end{proof}
\subsection{Proofs of theorems \ref{Upper bound of Algorithm 1} and \ref{Upper bound of Algorithm 2}}
\subsubsection{Preliminary lemmas} We begin with the following lemmas. Lemma \ref{Operator concentration} is a concentration inequality on the variable $\|\boldsymbol{{B}}^\ell\|_{\text{op}}$, which results from a concentration inequality on subgaussian processes. Lemma \ref{Neumann} states that $\|(\boldsymbol{K}_\de^\ell)^{-1}\|_{\text{op}}$ behaves as $\|(\boldsymbol{K}^\ell)^{-1}\|_{\text{op}}$ on a set with large probability.
Finally, Lemma \ref{Deviations} establishes deviations bounds on the variables $\boldsymbol{\zeta}_\ell-\check{\boldsymbol{f}_\ell}$ which will be useful throughout the proofs of Theorem \ref{Upper bound of Algorithm 1} and Theorem \ref{Upper bound of Algorithm 2}. 
\begin{lemma}
\label{Operator concentration}
There exists $\beta_0$, $c_0$ independent from $\ell\geq 0$, such that, for all $\ell\geq 0$, for all $t\geq \beta_0$,
$$\PP\Big(\frac{1}{\sqrt{(\ell\vee 1)\log (\ell\vee 2)}}\|\boldsymbol{{B}}^\ell\|_{\text{op}}>t\Big)\leq \exp(-c_0t^2)$$
This readily entails the following moments control, available for all $\ell\geq 0$, $p\geq1$ 
$$\E\|\boldsymbol{{B}}^\ell\|_{\text{op}}^p\lesssim \big(\ell\log \ell\big)^{p/2}\vee 1  $$
\end{lemma}

\begin{proof}
The proof is a slight modification of \citet[Theorem 1]{Meckes}, to which we refer for a complete study. Lemma \ref{Operator concentration} is trivially satisfied if $\ell=0,1$, hence we will suppose that $\ell\geq 2$. From Proposition \ref{Toeplitz norm}, we derive that 
$$\E\|\boldsymbol{{B}}^\ell\|_{\text{op}}\leq \E\|T\big((\boldsymbol{b})_k\big)\|_{\text{op}}=\E\sup_{x\in[0,1]}|Y_x|,\hspace{0.5cm}Y_x\overset{\Delta}{=}\Sum_{k=0}^{\ell} \boldsymbol{b}_\ell e^{2i\pi kx} $$
 We claim the two following facts:
\begin{itemize}
\item Let $a_0,...,a_\ell \in \R$. There exists $c\geq 0$ such that ofor all $t>0$, 
\begin{equation}
\label{ineq}
\PP\big(\big|\Sum_{k=0}^{\ell} a_k \boldsymbol{b}_k\big|>t\big) \leq \exp\big(\frac{-ct^2}{\sum_{k=0}^{\ell} a_k^2}\big)
\end{equation}
\item $d(x,y)\overset{\Delta}{=}\sqrt{\E|Y_x-Y_y|^2} \leq 4 \ell^{3/2}|x-y|\wedge 2\sqrt{\ell}$
\end{itemize}
The first point is readily verified since $(\boldsymbol{b}_k)_{k\leq \ell}$ is a standard Gaussian vector, while the second point directly results from the bound $$\big|e^{2i\pi k x}-e^{2i\pi k y} \big|\leq 2\wedge 2\pi k |x-y| \text{ for all } x,y\in[0,1],\, k\geq 0$$
A direct application of Dudley's entropy bound (\citet[Proposition 2.1]{Talagrand}) now entails 
$$\E\sup_{x\in[0,1]}|Y_x|\lesssim (\ell\log \ell)^{1/2}$$
(see \citet{Meckes} for the rest of the proof). The deviation bound is now a consequence of \citet[Lemma 5.3]{Talagrand}. Indeed, for all $x\in[0,1]$, 
$$E|Y_x|^2= \E\big|\Sum_{k=0}^{\ell} \boldsymbol{b}_k e^{2i\pi kx}\big|^2 \lesssim \ell$$
which ends the proof.
\end{proof}

\begin{lemma}
\label{Neumann}
Let $\ell\geq 0$, $a_\ell = \rho O_{\ell,\de}$ for some $0 < \rho < \tfrac{1}{2}$. Note $\boldsymbol{\gamma}_\de (z)=\Sum_{\ell \geq 0} \boldsymbol{\gamma}_{k,\de}z^k $ the power series associated to $(\boldsymbol{K}_{\delta}^\ell)^{-1}$. On $\boldsymbol{A}_\ell \overset{\Delta}{=} \{\|(\boldsymbol{K}_{\delta}^\ell)^{-1}\|_{\text{op}}\leq O_{\ell,\de}^{-1}\}$ and $\boldsymbol{B}_\ell\overset{\Delta}{=}\{\|\delta \boldsymbol{ {B}^\ell}\|_{\text{op}} \leq a_\ell\}$, the following inequalities hold 
\begin{align}
\label{Neumann op}
\|(\boldsymbol{K}_\de^\ell)^{-1}\|_{\text{op}}\leq \frac{\rho}{1-\rho} \|(\boldsymbol{K}^\ell)^{-1}\|_{\text{op}} &\text{ and } \|(\boldsymbol{K}^\ell)^{-1}\|_{\text{op}}\leq (1-\rho)^{-1} \|(\boldsymbol{K}_\de^\ell)^{-1}\|_{\text{op}}
\\\label{Neumann HS} \|(\boldsymbol{K}_\de^\ell)^{-1}\|_{\text{HS}}\leq \frac{\rho}{1-\rho} \|(\boldsymbol{K}^\ell)^{-1}\|_{\text{HS}} &\text{ and } \|(\boldsymbol{K}^\ell)^{-1}\|_{\text{HS}}\leq (1-\rho)^{-1} \|(\boldsymbol{K}_\de^\ell)^{-1}\|_{\text{HS}}
\\\label{Neumann norm} \Sum_{k=0}^\ell \boldsymbol{\gamma}_{k,\de} ^2 \leq \frac{\rho}{1-\rho} \Sum_{k=0}^\ell \boldsymbol{\gamma}_k ^2 &\text{ and } \Sum_{k=0}^\ell \boldsymbol{\gamma}_k ^2 \leq (1-\rho)^{-1} \Sum_{k=0}^\ell \boldsymbol{\gamma}_{k,\de} ^2
\end{align}
\end{lemma}

\begin{proof}
First, we have
$$
\big(\boldsymbol{K}^\ell\big)^{-1}=\big(\boldsymbol{K}_\de^\ell-\de\boldsymbol{{B}}^\ell\big)^{-1}=\big(\boldsymbol{I}-\de (\boldsymbol{K}_\de^\ell)^{-1}\boldsymbol{{B}}^\ell\big)^{-1}\big(\boldsymbol{K}_\de^\ell\big)^{-1}
$$

On $\boldsymbol{A}_\ell \cap\boldsymbol{B}_\ell$, since $a_\ell$ satisfies $O_{\ell,\de}^{-1}\, a_\ell = \rho < \tfrac{1}{2}$, by a usual Neumann series argument, we have
\begin{align}
\notag\|\big(\boldsymbol{K}^\ell\big)^{-1}\|_{\text{op}}&=\Big\| \Big[\Sum_{k\geq 0} \big(-\de(\boldsymbol{K}_\de^\ell)^{-1}\boldsymbol{{B}}^\ell\big)^k\Big]\big(\boldsymbol{K}_\de^\ell\big)^{-1}\Big\|_{\text{op}}
\\\notag&\leq \Big[\Sum_{k\geq 0} \de^k\|\big(\boldsymbol{K}_\de^\ell\big)^{-1}\|_{\text{op}}^k\|\boldsymbol{{B}}^\ell\|_{\text{op}}^k\Big]\|\big(\boldsymbol{K}_\de^\ell\big)^{-1}\|_{\text{op}}
\\\notag&\leq \Big[\Sum_{k\geq 0} \rho^k\Big]\|\big(\boldsymbol{K}_\de^\ell\big)^{-1}\|_{\text{op}}
\\\label{Sens 1}&\leq (1-\rho)^{-1} \|\big(\boldsymbol{K}_\de^\ell\big)^{-1}\|_{\text{op}}
\end{align}
Secondly, we have 
$$
\big(\boldsymbol{K}_\de^\ell\big)^{-1}=\big(\boldsymbol{K}^\ell+\de\boldsymbol{{B}}^\ell\big)^{-1}
=\big(\boldsymbol{I}+\de (\boldsymbol{K}^\ell)^{-1}\boldsymbol{{B}}^\ell\big)^{-1}\big(\boldsymbol{K}^\ell\big)^{-1}
$$
Moreover, thanks to \eqref{Sens 1}, on $\boldsymbol{A}_\ell \cap\boldsymbol{B}_\ell$, we have
\begin{align}
\label{ineq1}
\|\de(\boldsymbol{K}^\ell)^{-1}\boldsymbol{{B}}^\ell\|_{\text{op}}\leq (1-\rho)^{-1} O_{\ell,\de}^{-1} \,a_\ell \leq \frac{\rho}{1-\rho}<1
\end{align}
So that we can now similarly derive
$$
\big\|\big(\boldsymbol{K}_\de^\ell\big)^{-1}\big\|_{\text{op}}\leq\frac{\rho}{1-\rho}\big\|\big(\boldsymbol{K}^\ell\big)^{-1}\big\|_{\text{op}} 
$$
This prooves \eqref{Neumann op}. The proofs of \eqref{Neumann HS} and \eqref{Neumann norm} follow the same lines, since $\|AB\|_{\text{HS}}\leq\|A\|_{\text{op}}\|B\|_{\text{HS}}$, and $\|Ab\|\leq \|A\|_{\text{op}} \|b\|$.
\end{proof}

\subsubsection{Proof of theorem \ref{Upper bound of Algorithm 1}}
\begin{lemma}
\label{Deviations}
Under Assumption \ref{DIP}, we have, for all $\ell \geq 0$,
\begin{align}
\label{Expectation bound} \E\Big[\Big|\langle (\boldsymbol{K}^\ell_\de)^{-1}\indi{\boldsymbol{A}_\ell}\indi{\boldsymbol{B}_\ell}\big(-\de \boldsymbol{{B}}^\ell\boldsymbol{f}^\ell+\ep\boldsymbol{\xi}_{L^{\textbf{I}}}\big),\boldsymbol{\phi}_\ell\rangle\Big|^q\Big]&\lesssim (\ell\vee 1)^{q\nu} (\ep \vee\de)^q
\\\label{Deviation inequality}\PP\Big(\Big|\langle (\boldsymbol{K}^\ell_\de)^{-1}\indi{\boldsymbol{A}_\ell}\indi{\boldsymbol{B}_\ell}\big(-\de \boldsymbol{{B}}^\ell\boldsymbol{f}^\ell+\ep\boldsymbol{\xi}_{L^{\textbf{I}}}\big),\boldsymbol{\phi}_\ell\rangle\Big|>S_{\ell,\ep}^{\textbf{I}}\Big)&\lesssim  \ep^{\tau^2}\vee \de^{\tau}
\end{align}
\end{lemma}

\begin{proof}
In order to prove Inequalities \eqref{Expectation bound} and \eqref{Deviation inequality}, it suffices to study the tails of the random variables $\Big|\langle (\boldsymbol{K}^\ell_\de)^{-1}\indi{\boldsymbol{A}_\ell}\indi{\boldsymbol{B}_\ell}\big(-\de \boldsymbol{{B}}^\ell\boldsymbol{f}^\ell+\ep\boldsymbol{\xi}_{L^{\textbf{I}}}\big),\boldsymbol{\phi}_\ell\rangle\Big|$. For convenience we will only treat the case where $\ell \geq 2$, otherwise the result follows by identical arguments. To this end, we study each term apart. On $\boldsymbol{A}_\ell\cap \boldsymbol{B}_\ell$, Lemma \ref{Neumann} and Assumption \ref{DIP} entail $$\|(\boldsymbol{K}^\ell_\de)^{-1}\|_{\text{op}}\leq \frac{Q\rho}{1-\rho} \ell^\nu$$
Thus, combining Assumption \ref{Design regularity} with the latter inequality, a brief conditionning argument readily yields
\begin{align*}
\PP\Big(\Big|\langle (\boldsymbol{K}^\ell_\de)^{-1}\boldsymbol{1}_{\boldsymbol{A}_\ell}\boldsymbol{1}_{\boldsymbol{B}_\ell}\ep\boldsymbol{\xi}_{L^{\textbf{I}}},\boldsymbol{\phi}_\ell\rangle\Big|>t\Big)\lesssim \exp\big(-\frac{t^2}{\ep^2 \ell^{2\nu}}\big)
\end{align*}
Let us study the second term.
 On $\boldsymbol{A}_\ell\cap\boldsymbol{B}_\ell$, we have 
$$\de(\boldsymbol{K}^\ell_\de)^{-1} \boldsymbol{{B}}^\ell=\Sum_{k\geq 1} \big(\de (\boldsymbol{K}^\ell)^{-1}\boldsymbol{{B}}^\ell\big)^k
= \de (\boldsymbol{K}^\ell)^{-1}\boldsymbol{{B}}^\ell+\Sum_{k\geq 2} \big(\de (\boldsymbol{K}^\ell)^{-1}\boldsymbol{{B}}^\ell\big)^k$$
Hence,
\begin{align}
\label{decomp}
\de(\boldsymbol{K}^\ell_\de)^{-1}\indi{\boldsymbol{A}_\ell}\indi{\boldsymbol{B}_\ell} \boldsymbol{{B}}^\ell\boldsymbol{f}^\ell=r_1+r_2
\end{align}
where 
\begin{align}
\label{r1r2}
\begin{cases}
r_1&= \langle \de (\boldsymbol{K}^\ell)^{-1}\boldsymbol{{B}}^\ell\boldsymbol{f}^\ell,\boldsymbol{\phi}_\ell \rangle
\\r_2&= \langle \big(\de (\boldsymbol{K}^\ell)^{-1}\boldsymbol{{B}}^\ell\big)^2 \big(\boldsymbol{I}+\de (\boldsymbol{K}^\ell)^{-1} \boldsymbol{B}^\ell\big)^{-1}\boldsymbol{f}^\ell,\boldsymbol{\phi}_\ell \rangle
\end{cases}
\end{align}
Let's now bound separately $r_1$ and $r_2$. We first apply equality \eqref{Toeplitz product and inverse} to get 
\begin{align*}
\langle (\boldsymbol{K}^\ell)^{-1}\boldsymbol{B}^\ell \boldsymbol{f}^\ell,\boldsymbol{\phi}_\ell \rangle&= \langle (\boldsymbol{K}^\ell)^{-1} \boldsymbol{f}^\ell,{^{\boldsymbol{t}}}\boldsymbol{B}^\ell\boldsymbol{\phi}_\ell \rangle
\\&=\langle (\boldsymbol{K}^\ell)^{-1} \boldsymbol{f}^\ell,(\boldsymbol{b}^\ell)' \rangle 
\end{align*}
where $(\boldsymbol{b}^\ell)_k'=(\boldsymbol{b}^\ell)_{\ell-k}$.
The result is a centred gaussian variable with variance 
$$\|\de (\boldsymbol{K}^\ell)^{-1}\boldsymbol{f}^\ell\|^2\leq \de^2 Q^2M^2 \ell^{2\nu}$$
which hence satisfies $$\PP(|r_1|>t)\lesssim \exp\big(\frac{-t^2}{\de^2\ell^{2\nu}}\big)$$
Let us study the term $r_2$. Since the maximal level $L$ verifies $L\leq \lambda (\de |\log \de|)^{-\frac{1}{\nu+1}}$, we have, for all $\ell\leq L$, $\de \ell^{\nu+1}\log \ell \lesssim 1$. We deduce that
\begin{align*}
\PP(|r_2| >t)&\leq \PP(\de^2 \ell^{2\nu} \|\boldsymbol{B}^\ell\|_{\text{op}}^2 >t) 
\\&\leq \PP(  \frac{1}{\ell \log \ell}\|\boldsymbol{B}^\ell\|_{\text{op}}^2> \de^{-2}\ell^{-2\nu-1} (\log \ell)^{-1} t)
\\&\lesssim \PP(  \frac{1}{\ell \log \ell}\|\boldsymbol{B}^\ell\|_{\text{op}}^2> \de^{-1}\ell^{-\nu} t)
\\&\lesssim \exp(-t\big(\de \ell^\nu \,\big)^{-1})\ind{t>\beta_0 \de \ell^\nu} + \ind{t\leq \beta_0 \de  \ell^\nu\,}
\end{align*}
inequality \eqref{Deviation inequality} directly follows, and inequality \eqref{Expectation bound} is now a direct application of the well known formula
$$\E[X^2]=\int_{t>0} 2t\PP(|X|>t)dt$$
\end{proof}

\begin{proof}[Proof of theorem \ref{Upper bound of Algorithm 1}]
We apply Parseval's formula to derive
\begin{align*}
\E\|\tilde{\boldsymbol{f}}^{\textbf{I}}-\boldsymbol{f}\|_2^2 = \Sum_{\ell\leq L^{\textbf{I}}}\E\langle \widetilde{\boldsymbol{f}}-\boldsymbol{f},\boldsymbol{\phi}_\ell \rangle^2 + \Sum_{\ell>L^{\textbf{I}}} \check{\boldsymbol{f}}_\ell^2
\end{align*}
The second term is easily handled. Remark first that, since  $s>1/2$, we have $\frac{2s}{\nu+1}>\frac{2s}{s+\nu+1/2}$ and we can write 
\begin{align*}
\Sum_{\ell>L^{\textbf{I}}} \check{\boldsymbol{f}}_\ell^2 &\leq \big(L^{\textbf{I}}\big)^{-2s}
\\&\leq \big(\ep\sqrt{|\log \ep|}\big)^{\frac{2s}{\nu+1}}\vee \big(\de|\log \de|\big)^{\frac{2s}{\nu+1}}
\\&\leq \big(\ep\sqrt{|\log \ep|}\big)^{\frac{4s}{2(s+\nu)+1}}\vee \big(\de|\log \de|\big)^{\frac{4s}{2(s+\nu)+1}}
\end{align*}
In order to lighten the notations, we will only consider the indexes $\ell \geq 2$ in the first term. This is of course not problematic, since an identical reasoning allows to bound the two remaining summands by the desired rates of convergence. We hence write the following decomposition
\begin{align*}
\Sum_{\ell\leq L^{\textbf{I}}} \E \langle \widetilde{\boldsymbol{f}}-\boldsymbol{f},\boldsymbol{\phi}_\ell\rangle^2&= \Sum_{\ell\leq L^{\textbf{I}}} \E (\boldsymbol{\zeta}_\ell-\check{\boldsymbol{f}}_\ell)^2\ind{|\boldsymbol{\zeta}_\ell|>S_{\ell,\ep}^{\textbf{I}}}+\E\Sum_{\ell\leq L^{\textbf{I}}} \check{\boldsymbol{f}}_\ell^2\ind{|\boldsymbol{\zeta}_\ell|\leq S_{\ell,\ep}^{\textbf{I}}}
\\&\lesssim I+II+III+IV
\end{align*}
where 
\begin{align*}
I&=\Sum_{\ell\leq L^{\textbf{I}}}\E \big(\boldsymbol{\zeta}_\ell-\check{\boldsymbol{f}}_\ell\big)^2\indi{\boldsymbol{A}_\ell}\indi{\boldsymbol{B}_\ell}\ind{|\boldsymbol{\zeta}_\ell|>S_{\ell,\ep}^{\textbf{I}}}
\\II&=\Sum_{\ell\leq L^{\textbf{I}}} \E \check{\boldsymbol{f}}_\ell^2\ind{|\boldsymbol{\zeta}_\ell|\leq S_{\ell,\ep}^{\textbf{I}}}\indi{\boldsymbol{A}_\ell}\indi{\boldsymbol{B}_\ell}
\\III&=\Sum_{\ell\leq L^{\textbf{I}}}\E \big(\boldsymbol{\zeta}_\ell-\check{\boldsymbol{f}}_\ell\big)^2\indi{\boldsymbol{A}_\ell}\indi{\boldsymbol{B}_\ell^c}\ind{|\boldsymbol{\zeta}_\ell|>S_{\ell,\ep}^{\textbf{I}}}+\Sum_{\ell\leq L^{\textbf{I}}} \E \check{\boldsymbol{f}}_\ell^2\ind{|\boldsymbol{\zeta}_\ell|\leq S_{\ell,\ep}^{\textbf{I}}}\indi{\boldsymbol{B}_\ell^c}
\\IV&=\Sum_{\ell\leq L^{\textbf{I}}}\E \check{\boldsymbol{f}}_\ell^2\indi{\boldsymbol{A}_\ell^c}\ind{|\boldsymbol{\zeta}_\ell|\leq S_{\ell,\ep}^{\textbf{I}}}
\end{align*}
\begin{proof}[$\bullet$ \textbf{Term I and II}]
On $\boldsymbol{A}_\ell$, we have 
\begin{align}
\boldsymbol{\zeta}_\ell-\check{\boldsymbol{f}}_\ell=\langle(\boldsymbol{K}^\ell_\de)^{-1}\big(-\de \boldsymbol{{B}}^\ell\boldsymbol{f}^\ell+\ep\boldsymbol{\xi}\big),\boldsymbol{\phi}_\ell\rangle
\end{align} 
Hence we can decompose further $I$ as 
{\small
\begin{align*}
I\lesssim& \Sum_{\ell\leq L^{\textbf{I}}} \E \langle(\boldsymbol{K}^\ell_\de)^{-1}\indi{\boldsymbol{A}_\ell}\indi{\boldsymbol{B}_\ell} \big(-\de\boldsymbol{{B}}^\ell\boldsymbol{f}^\ell+\ep\boldsymbol{\xi}\big),\boldsymbol{\phi}_\ell \rangle^2\ind{|\boldsymbol{\zeta}_\ell|>{S_{\ell,\ep}^{\textbf{I}}}}
\\\lesssim& \Sum_{\ell\leq L^{\textbf{I}}} \E \langle(\boldsymbol{K}^\ell_\de)^{-1}\indi{\boldsymbol{A}_\ell}\indi{\boldsymbol{B}_\ell} \big(-\de\boldsymbol{{B}}^\ell\boldsymbol{f}^\ell+\ep\boldsymbol{\xi}\big),\boldsymbol{\phi}_\ell\rangle^2\ind{|\boldsymbol{\zeta}_\ell|>{S_{\ell,\ep}^{\textbf{I}}}}\ind{|\check{\boldsymbol{f}}_\ell|\geq {S_{\ell,\ep}^{\textbf{I}}}/2}
\\&+\Sum_{\ell\leq L^{\textbf{I}}} \E \langle (\boldsymbol{K}^\ell_\de)^{-1}\indi{\boldsymbol{A}_\ell}\indi{ \boldsymbol{B}_\ell} \big(-\de\boldsymbol{{B}}^\ell\boldsymbol{f}^\ell+\ep\boldsymbol{\xi}\big),\boldsymbol{\phi}_\ell\rangle ^2\ind{|\boldsymbol{\zeta}_\ell|>{S_{\ell,\ep}^{\textbf{I}}}}\ind{|\check{\boldsymbol{f}}_\ell|<{S_{\ell,\ep}^{\textbf{I}}}/2}
\\&\overset{\Delta}{=}V+VI
\end{align*}
}Let us first treat the term $VI$. From Lemmas \ref{Neumann} and \ref{Deviations} and Cauchy-Schwarz inequality, we derive
\begin{align*}
VI&\leq \Sum_{\ell\leq L^{\textbf{I}}}\E \Big[\langle(\boldsymbol{K}^\ell_\de)^{-1}\indi{\boldsymbol{A}_\ell}\indi{ \boldsymbol{B}_\ell} \big(\de\boldsymbol{{B}}^\ell\boldsymbol{f}^\ell+\ep\boldsymbol{\xi}\big),\boldsymbol{\phi}_\ell \rangle ^4 \Big]^{1/2}
\\&\hspace{6cm}.\PP(|\boldsymbol{\zeta}_\ell-\check{\boldsymbol{f}}_\ell|>{S_{\ell,\ep}^{\textbf{I}}})^{1/2}
\\&\lesssim \Sum_{\ell\leq L^{\textbf{I}}}\big[(\de\vee \ep)^4 \ell^{4\nu} \big]^{1/2} \big(\de^{\tau/2}\vee\ep^{\tau^2/2}\big)
\end{align*}
which is less than the desired bound for $\tau$ large enough.
As for term $V$, we split it in two and write
\begin{align*}
V\leq& \Sum_{\ell\leq L^{\textbf{I}}} \E \langle(\boldsymbol{K}^\ell_\de)^{-1}\indi{\boldsymbol{A}_\ell}\indi{ \boldsymbol{B}_\ell} \big(\de\boldsymbol{{B}}^\ell\boldsymbol{f}^\ell+\ep\boldsymbol{\xi}\big),\boldsymbol{\phi}_\ell \rangle^2\ind{|\check{\boldsymbol{f}}_\ell|\geq {S_{\ell,\ep}^{\textbf{I}}}/2}
\\\leq& \Sum_{\ell\leq L^{\textbf{I}}} \E \langle (\boldsymbol{K}^\ell_\de)^{-1}\indi{\boldsymbol{A}_\ell}\indi{ \boldsymbol{B}_\ell} \de\boldsymbol{{B}}^\ell\boldsymbol{f}^\ell,\boldsymbol{\phi}_\ell\rangle^2\ind{|\check{\boldsymbol{f}_\ell}|\geq \ell^\nu \de|\log \de|/2}
\\&+\Sum_{\ell\leq L^{\textbf{I}}} \E \langle(\boldsymbol{K}^\ell_\de)^{-1}\indi{\boldsymbol{A}_\ell}\indi{ \boldsymbol{B}_\ell} \ep\boldsymbol{\xi},\boldsymbol{\phi}_\ell\rangle^2\ind{|\check{\boldsymbol{f}_\ell}|\geq \ell^{\nu} \ep\sqrt{|\log \ep|}/2}
\\ \leq& \Sum_{\ell\leq L^{\textbf{I}}} \de^2 \ell^{2\nu} \big(\check{\boldsymbol{f}}_\ell^2\big(\ell^\nu\, \de|\log \de|\big)^{-2}\wedge 1\big)
\\& \hspace{2.5cm}\vee\Sum_{\ell\leq L^{\textbf{I}}} \ep^2 \ell^{2\nu}  \big(\check{\boldsymbol{f}}_\ell^2\big(\ell^{\nu} \ep\sqrt{|\log \ep|}\big)^{-2}\wedge 1\big)
\end{align*}
Note $\ell_\de=\big(\de|\log \de|\big)^{\frac{-2}{2(s+\nu)+1}}$ and write
\begin{align*}
\Sum_{\ell\leq L^{\textbf{I}}} \de^2 \ell^{2\nu} \big(\check{\boldsymbol{f}}_\ell^2\big(\ell^\nu \de|\log \de|\big)^{-2}\wedge 1\big)&\leq \Sum_{\ell\leq \ell_\de}\de^2  \ell^{2\nu}+\Sum_{\ell>\ell_\de} \check{\boldsymbol{f}}_\ell^2 |\log \de|^{-2}
\\&\lesssim \big(\de|\log \de|\big)^{\frac{4s}{2(s+\nu)+1}}
\end{align*}
The $\ep$-term is treated similarly by taking $\ell_\ep=\big(\ep \sqrt{|\log \ep|}\big)^{\frac{-2}{2s+2\nu+1}}$ and leads to the desired convergence rate.
As for the term $II$, a similar reasonning leads to
\begin{align*}
II\leq& \Sum_{\ell\leq L^{\textbf{I}}} \E \check{\boldsymbol{f}}_\ell^2\ind{|\boldsymbol{\zeta}_\ell|\leq {S_{\ell,\ep}^{\textbf{I}}}}\ind{|\check{\boldsymbol{f}_\ell}|\leq 2{S_{\ell,\ep}^{\textbf{I}}}}+\Sum_{\ell\leq L^{\textbf{I}}} \E \check{\boldsymbol{f}}_\ell^2\ind{|\boldsymbol{\zeta}_\ell|\leq {S_{\ell,\ep}^{\textbf{I}}}}\ind{|\check{\boldsymbol{f}_\ell}|> 2{S_{\ell,\ep}^{\textbf{I}}}}
\\\overset{\Delta}{=}&VII+VIII
\end{align*}
The term $VIII$ is handled as the term $VI$. Indeed,
$$
VIII\leq \Sum_{\ell\leq L^{\textbf{I}}}  \check{\boldsymbol{f}}_\ell^2\PP\big(|\boldsymbol{\zeta}_\ell-\check{\boldsymbol{f}}_\ell|> {S_{\ell,\ep}^{\textbf{I}}}\big)
\leq \Sum_{\ell\leq L^{\textbf{I}}}\check{\boldsymbol{f}}_\ell^2 (\ep^{\tau^2}\vee\de^{\tau})
$$
which is less than the desired rate for $\tau$ large enough. Finally, we have
\begin{align*}
VII\leq&\Sum_{\ell\leq L^{\textbf{I}}} \E \check{\boldsymbol{f}}_\ell^2\ind{|\check{\boldsymbol{f}_\ell}|\leq 2{S_{\ell,\ep}^{\textbf{I}}}}
\\\leq& \Sum_{\ell\leq L^{\textbf{I}}} \E \check{\boldsymbol{f}}_\ell^2\ind{|\check{\boldsymbol{f}_\ell}|\leq 2 \tau \ell^{\nu} \ep \sqrt{|\log \ep|} }+\Sum_{\ell\leq L^{\textbf{I}}} \E \check{\boldsymbol{f}}_\ell^2\ind{|\check{\boldsymbol{f}_\ell}|\leq 2 \tau \ell^\nu \de |\log \de| }
\\ \lesssim & \Sum_{\ell\leq \ell_\ep} \ell^{2\nu} \ep^2 |\log \ep|+\Sum_{\ell>\ell_\ep}\check{\boldsymbol{f}}_\ell^2+\Sum_{\ell\leq \ell_\de} \ell^{2\nu} \de^2 |\log \de|+\Sum_{\ell>\ell_\de}\check{\boldsymbol{f}}_\ell^2
\\&\lesssim  \big(\ep\sqrt{|\log \ep|}\big)^{\frac{4s}{2(s+\nu)+1}}\vee\big(\de|\log \de|^2\big)^{\frac{4s}{2(s+\nu)+1}}
\end{align*}
\end{proof}
\begin{proof}[$\bullet$ \textbf{Term III}]We have
\begin{align*}
III&\leq \Sum_{\ell\leq L^{\textbf{I}}}\E\Big( \big(\boldsymbol{\zeta}_\ell-\check{\boldsymbol{f}}_\ell\big)^2\indi{\boldsymbol{A}_\ell}+\check{\boldsymbol{f}}_\ell^2\Big)\indi{\boldsymbol{B}_\ell^c}
\\&\leq\Sum_{\ell\leq L^{\textbf{I}}}\E\Big[ \big(\boldsymbol{\zeta}_\ell-\check{\boldsymbol{f}}_\ell\big)^4\indi{\boldsymbol{A}_\ell}\Big]^{1/2}\PP\big(\boldsymbol{B}_\ell^c\big)^{1/2}+\Sum_{\ell\leq L^{\textbf{I}}}\check{\boldsymbol{f}}_\ell^2\PP(\boldsymbol{B}_\ell^c)  
\end{align*}
Moreover, Lemma \ref{Operator concentration} entails
\begin{align*}
\PP\big(\boldsymbol{B}_\ell^c\big)&\leq \de^{\kappa^2\rho^2}
\end{align*}
for all $\ell\geq 1$. It is hence clear that for $\kappa$ large enough, the term $III$ is less than the announced rate.
\end{proof}
\begin{proof}[$\bullet$ \textbf{Term IV}]We claim that
\begin{align*}
\ind{\boldsymbol{A}_\ell^c}\leq \ind{\|(\boldsymbol{K}^\ell)^{-1}\|_{\text{op}}\geq O_{\ell,\de}^{-1}/2}+\ind{\|\de\boldsymbol{{B}}^\ell\|_{\text{op}}\geq O_{\ell,\de}}
\end{align*}
for all $\ell \geq 0$ (see \citet{DHPV}, Lemma 5.3). Hence, 
\begin{align*}
IV&\leq \Sum_{\ell\leq L^{\textbf{I}}} \E{\check{\boldsymbol{f}}_\ell}^2
\big( \ind{\|(\boldsymbol{K}^\ell)^{-1}\|_{\text{op}}\geq O_{\ell,\de}^{-1}/2}+\ind{\|\de\boldsymbol{{B}}^\ell\|_{\text{op}}\geq O_{\ell,\de}}\big)
\\&\overset{\Delta}{=} VIII+IX
\end{align*}
Since $\|(\boldsymbol{K}^\ell)^{-1}\|_{\text{op}}\leq Q_2 \ell^{\nu}$, we have $\{\|(\boldsymbol{K}^\ell)^{-1}\|_{\text{op}}\geq O_{\ell,\de}^{-1}/2\}\subset \{\ell^{\nu+1/2} \sqrt{\log \ell} \geq c \big(\de|\log \de|\big)^{-1}\}$
where $c$ is a constant depending only on $Q_2$ and $\kappa$. Hence
\begin{align*}
VIII&\leq \Sum_{\ell\geq c \big(\de|\log \de|^{3/2}\big)^{-\frac{2}{2\nu+1}}}  \check{\boldsymbol{f}}_\ell^2
\\&\lesssim \big(\de|\log \de|\big)^{\frac{4s}{2\nu+2s+1}}
\end{align*}
As for $IX$, a quick application of \ref{Operator concentration} entails
$$\PP(\|\de\boldsymbol{{B}}^\ell\|_{\text{op}}\geq O_{\ell,\de})\leq \de^{\kappa^2}$$
,so that
$$
IX\leq \Sum_{\ell\leq L^{\textbf{I}}} {\check{\boldsymbol{f}}_\ell}^2 \de^{\kappa^2} \lesssim \de^{\kappa^2}
$$
which is less than the announced rate for $\kappa$ large enough.
\end{proof}
It remains to put together the bounds of the four terms above to get the desired rates of convergence in theorem \ref{Upper bound of Algorithm 1}.
\end{proof}

\subsubsection{Proof of theorem \ref{Upper bound of Algorithm 2}}

\begin{lemma}
\label{Deviations II}
Note, for $\ell \geq 0$, $\widebar{S_{\ell,\ep}^{\textbf{II}}}\overset{\Delta}{=}(\ell\vee 1)^{\nu-1/2}\big(\tau_{sig}\ep \sqrt{|\log \ep|} \vee \tau_{op}  \de|\log \de| \big)$. Under Assumption \ref{DIP 2},  we have, for all $\ell \geq 0$, for all $q\geq 0$,
\begin{align}
\label{Expectation bound II} \E\Big[\big|\langle(\boldsymbol{K}^\ell_\de)^{-1}\indi{\boldsymbol{A}_\ell}\indi{\boldsymbol{B}_\ell}(-\de \boldsymbol{B}^\ell\boldsymbol{f}^\ell+\ep\boldsymbol{\xi}_{L^{\textbf{II}}}\big),\boldsymbol{\phi}_\ell \rangle \big|^q\Big]&\lesssim  (\ell\vee 1)^{q(\nu-1/2)}(\ep \vee \de )^q
\\\label{Deviation inequality II}\PP\Big(\big|\langle(\boldsymbol{K}^\ell_\de)^{-1}\indi{\boldsymbol{A}_\ell}\indi{\boldsymbol{B}_\ell}(-\de \boldsymbol{B}^\ell\boldsymbol{f}^\ell+\ep\boldsymbol{\xi}_{L^{\textbf{II}}}\big),\boldsymbol{\phi}_\ell \rangle \big|>\widebar{S_{\ell,\ep}^{\textbf{II}}}\Big)&\lesssim  \ep^{\tau^2}\vee \de^{\tau}
\end{align}
\end{lemma}

\begin{proof}
The proof is very similar to Lemma \ref{Deviations}, whence we will just mention the notable changes compared to it. Once more, we shall only treat the case $\ell \geq 2$. First, we have
$$\langle (\boldsymbol{K}^\ell_\de)^{-1}\boldsymbol{1}_{\boldsymbol{A}_\ell}\boldsymbol{1}_{\boldsymbol{B}_\ell}\ep\boldsymbol{\xi},\boldsymbol{\phi}_\ell\rangle=\langle \ep\boldsymbol{\xi}_{L^{\textbf{II}}},^t(\boldsymbol{K}^\ell_\de)^{-1}\boldsymbol{1}_{\boldsymbol{A}_\ell}\boldsymbol{1}_{\boldsymbol{B}_\ell}\boldsymbol{\phi}_\ell\rangle=\ep \langle\boldsymbol{\xi}_{L^{\textbf{II}}}, ((1/\boldsymbol{\dot{g}}_\de)^\ell)'\rangle$$
so that a brief contitionning argument, combined with \eqref{Neumann norm} and Assumption \ref{Design regularity} entails
$$\PP(|\langle (\boldsymbol{K}^\ell_\de)^{-1}\boldsymbol{1}_{\boldsymbol{A}_\ell}\boldsymbol{1}_{\boldsymbol{B}_\ell}\ep\boldsymbol{\xi}_{L^{\textbf{II}}},\boldsymbol{\phi}_\ell\rangle|>t)\lesssim \exp(\frac{-t^2}{\ep^2 \ell^{2\nu-1} })$$
In order to treat the term $\PP(|\langle (\boldsymbol{K}^\ell_\de)^{-1}\boldsymbol{1}_{\boldsymbol{A}_\ell}\boldsymbol{1}_{\boldsymbol{B}_\ell}\de\boldsymbol{B}^\ell\boldsymbol{f}^\ell,\boldsymbol{\phi}_\ell\rangle|>t) $, we first establish a useful result for the sequel: if $\boldsymbol{g}$ satisfies \eqref{DIP 2.1}, then $$\|(\boldsymbol{K}^\ell)^{-1}\boldsymbol{f}^\ell\|=\|T_\ell(\boldsymbol{f})(1/\boldsymbol{\dot{g}})^\ell\|\leq \|T_\ell(\boldsymbol{f})\|_{\text{op}} \|(1/\boldsymbol{\dot{g}})^\ell\| $$
Furthermore, thanks to Proposition \ref{Toeplitz norm} we have $$\|T_\ell(\boldsymbol{f})\|_{\text{op}}\leq \Sum_{\ell \geq 0} |\check{\boldsymbol{f}}_\ell| \leq \Sum_{\ell \geq 0} \ell^{2s}|\check{\boldsymbol{f}}_\ell|^2 \Sum_{\ell \geq 0} \ell^{-2s}\lesssim 1$$
since $\boldsymbol{f}\in \mathcal{W}^s(M)$ and $s>1/2$.
We derive that \begin{align}
\label{Useful result}
\|(\boldsymbol{K}^\ell)^{-1}\boldsymbol{f}^\ell\|^2 \lesssim \ell^{2\nu-1}
\end{align}
Let us now bound the term of interest. Once more, we decompose it as $r_1+r_2$ where $r_1$ and $r_2$ are defined in \eqref{r1r2}. We now apply Proposition \ref{Toeplitz product and inverse} and \eqref{ineq1} and derive
\begin{align*}
\PP\Big(|r_1|>t\Big)&=\PP\big(|\langle \de  (\boldsymbol{K}^\ell)^{-1}\boldsymbol{{B}}^\ell \indi{\boldsymbol{A}_\ell}\indi{\boldsymbol{B}_\ell}\boldsymbol{f}^\ell, \boldsymbol{\phi}_\ell \rangle|>t\big)
\\&\leq \PP\big(|\langle \de  (\boldsymbol{K}^\ell)^{-1}\boldsymbol{f}^\ell, {^{\boldsymbol{t}}}\boldsymbol{{B}}^\ell\boldsymbol{\phi}_\ell \rangle|>t\big)
\end{align*}
The latter is a gaussian random variable with variance $\de^2\|(\boldsymbol{K}^\ell)^{-1}\boldsymbol{f}^\ell\|^2\lesssim \de^2 \ell^{2\nu-1}$ where we used \eqref{Useful result}.
Turning to the term $r_2$, we apply Proposition \ref{Toeplitz product and inverse} to derive
\begin{align*}
\PP\Big(|r_2|>t\Big)&=\PP\big(\big|\langle \de \big(\boldsymbol{{B}}^\ell\big)^2 (\boldsymbol{K}_\de^\ell)^{-1}\indi{\boldsymbol{A}_\ell}\indi{\boldsymbol{B}_\ell} \boldsymbol{f}^\ell , ^t(\boldsymbol{K}^\ell)^{-1}\boldsymbol{\phi}_\ell \rangle\big|>t\big)
\\&\lesssim \PP\big(\de \|\boldsymbol{{B}}^\ell\|^2_{\text{op}}\|(\boldsymbol{K}_\de^\ell)^{-1}\boldsymbol{f}^\ell\|\|^t(\boldsymbol{K}^\ell)^{-1}\boldsymbol{\phi}_\ell\|\indi{\boldsymbol{A}_\ell}\indi{\boldsymbol{B}_\ell}>t\big)
\end{align*}
We now apply \eqref{Neumann norm} and \eqref{Useful result} to get
$$\|(\boldsymbol{K}_\de^\ell)^{-1}\boldsymbol{f}^\ell\|\|^t(\boldsymbol{K}^\ell)^{-1}\boldsymbol{\phi}_\ell\|\lesssim \ell^{2\nu-1}$$
Hence,
\begin{align*}
\PP\Big(|r_2|>t\Big)&\lesssim \PP\big(\de^2 \ell^{2\nu-1} \|\boldsymbol{{B}}^\ell\|_{\text{op}}^2 \indi{\boldsymbol{A}_\ell}\indi{\boldsymbol{B}_\ell} >t\big) 
\\&\lesssim \PP\big(\frac{1}{\ell \log \ell}\|\boldsymbol{{B}}^\ell\|_{\text{op}}^2 \indi{\boldsymbol{A}_\ell}\indi{\boldsymbol{B}_\ell} >t (\de^2 \ell^{2\nu} \log \ell)^{-1}\big)
\end{align*}
Let us take a look back to Lemma \ref{Neumann}. On $\boldsymbol{A}_\ell\cap\boldsymbol{B}_\ell$ we have prooved that
$$\|(\boldsymbol{K}_\de^\ell)^{-1}\|_{\text{HS}} \geq  (1-\rho) Q_1 \ell^\nu $$
so that $\boldsymbol{A}_\ell\cap \boldsymbol{B}_\ell \subset\Big\{\de \ell^{\nu+1/2}\log \ell \lesssim 1\Big\}$. We deduce
\begin{align*}
\PP(|r_2|>t)&\lesssim \PP\Big(\frac{1}{\ell \log \ell} \|\boldsymbol{B}^\ell\|_{\text{op}}^2>t(\de \ell^{\nu-1/2})^{-1})
\\&\lesssim \exp\big( \frac{-t}{\de \ell^{\nu-1/2}}\big)\ind{t> \beta_0^2\de \ell^{\nu-1/2} }+\ind{t\leq \beta_0^2\de \ell^{\nu-1/2} }
\end{align*}
The end of the proof is identical to Lemma \ref{Deviations}.

%
%
\end{proof}

\begin{proof}[Proof of Theorem \ref{Upper bound of Algorithm 2}] The proof is very similar to Theorem \ref{Upper bound of Algorithm 1}, whence we will just emphasize the notable changes compared to it.
First, we apply Parseval's formula to derive
\begin{align*}
\E\|\tilde{\boldsymbol{f}}^{\textbf{II}}-\boldsymbol{f}\|_2^2 = \Sum_{\ell\leq L^{\textbf{II}}}\E \langle \widetilde{\boldsymbol{f}}^{\textbf{II}}-{\boldsymbol{f}},\boldsymbol{\phi}_\ell\rangle^2 + \Sum_{\ell>L^{\textbf{II}}} \check{\boldsymbol{f}}_\ell^2
\end{align*}
The second term is easily handled, since
\begin{align*} 
\Sum_{\ell>L^{\textbf{II}}} \check{\boldsymbol{f}}_\ell^2 &\leq (L^{\textbf{II}})^{-2s}
\\&\leq \big(\ep\sqrt{|\log \ep|}\vee \de|\log \de|\big)^{2s}
\\&\leq \big(\ep\sqrt{|\log \ep|}\big)^{\frac{2s}{s+\nu}}\vee \big(\de|\log \de|\big)^{\frac{2s}{s+\nu}}
\end{align*}
To bound the first sum, we write the following decomposition
\begin{align*}
\Sum_{\ell\leq L^{\textbf{II}}} \E\|\widetilde{\boldsymbol{f}}^{\textbf{II}}-\boldsymbol{f}\|^2&= \Sum_{\ell\leq L^{\textbf{II}}} \E\big(\boldsymbol{\zeta}_\ell -\check{\boldsymbol{f}}_\ell\big)^2\ind{|\boldsymbol{\zeta}_\ell|>S_{\ell,\ep}^{\textbf{II}}}+\E\Sum_{\ell\leq L} \check{\boldsymbol{f}}_\ell^2\ind{|\boldsymbol{\zeta}_\ell|\leq S_{\ell,\ep}^{\textbf{II}}}
\\&\lesssim I+II+III+IV
\end{align*}
where 
\begin{align*}
I&=\Sum_{\ell\leq L^{\textbf{II}}}\E \big(\boldsymbol{\zeta}_\ell-\check{\boldsymbol{f}}_\ell\big)^2\indi{\boldsymbol{A}_\ell}\indi{\boldsymbol{B}_\ell}\ind{|\boldsymbol{\zeta}_\ell|>S_{\ell,\ep}^{\textbf{II}}}
\\II&=\Sum_{\ell\leq L^{\textbf{II}}} \E \check{\boldsymbol{f}}_\ell^2\ind{|\boldsymbol{\zeta}_\ell|\leq S_{\ell,\ep}^{\textbf{II}}}\indi{\boldsymbol{A}_\ell}\indi{\boldsymbol{B}_\ell}
\\III&=\Sum_{\ell\leq L^{\textbf{II}}}\E \big(\boldsymbol{\zeta}_\ell-\check{\boldsymbol{f}}_\ell\big)^2\indi{\boldsymbol{A}_\ell}\indi{\boldsymbol{B}_\ell^c}\ind{|\boldsymbol{\zeta}_\ell|>S_{\ell,\ep}^{\textbf{II}}}+\Sum_{\ell\leq L^{\textbf{II}}} \E \check{\boldsymbol{f}}_\ell^2\ind{|\boldsymbol{\zeta}_\ell|\leq S_{\ell,\ep}^{\textbf{II}}}\indi{\boldsymbol{B}_\ell^c}
\\IV&=\Sum_{\ell\leq L^{\textbf{II}}}\E \check{\boldsymbol{f}}_\ell^2\indi{\boldsymbol{A}_\ell^c}\ind{|\boldsymbol{\zeta}_\ell|\leq S_{\ell,\ep}^{\textbf{II}}}
\end{align*}
Thanks to Lemma \ref{Neumann} and the definition of  $S_{\ell,\ep}^{\textbf{II}}$, we have $$\frac{Q_1}{1-\rho}\widebar{S_{\ell,\ep}^{\textbf{II}}}\leq S_{\ell,\ep}^{\textbf{II}} \leq \frac{Q_2\rho}{1-\rho}\widebar{S_{\ell,\ep}^{\textbf{II}}}$$ on  $\boldsymbol{A}_\ell\cap \boldsymbol{B}_\ell$. Thus, the Terms I and II can be treated identically to the preceding proof and yield the desired rates of convergence. The terms III  and IV are treated exactly as in the preceding proof.
\end{proof}

\subsection{Proof of theorem \ref{Lower bound}}
\begin{proof}
The lower bound will not decrease for increasing noise levels $\de$ and $\ep$, whence it suffices to provide the case $\de=0$ and the case $\ep=0$ separately. In the sequel, $c_i$ will denote a positive constant to be adjusted later and $L$ will play the role of a maximal level. Also, we will note $\boldsymbol{K}_\nu$ (resp. $\boldsymbol{\dot{g}}_\nu$) the operator (resp. the function) associated with the Laurent serie $(1-z)_L^\nu$. The function $\boldsymbol{\dot{g}}_{-1/2}$ will play an essential role in the sequel. Unfortunately it is not square integrable. We thus begin with a preliminary lemma, which states that a minor modification corrects this defect.
\begin{lemma}
\label{Lemma lower bound}
Let $\boldsymbol{h}$ be the function associated to the Laurent serie $\Sum_{\ell\geq 0} \frac{(-1)^\ell}{\log (\ell\vee2)} \binom{-1/2}{\ell} z^\ell$. Then $\boldsymbol{h}$ is square integrable. Furthermore, for all $\nu\geq 0$, for all $\ell\leq L$, $$\frac{(\ell\vee 1)^{\nu-1/2}}{\log (\ell\vee 2) } \lesssim |\langle \boldsymbol{K}_{-\nu} \boldsymbol{h},\boldsymbol{\phi}_\ell \rangle|\lesssim  (\ell\vee 1)^{\nu-1/2}  $$
\end{lemma}
\begin{proof}[Proof of Lemma \ref{Lemma lower bound}] $\boldsymbol{h}$ is trivially squared integrable thanks to \eqref{Binom equiv} and Parseval's formula. Now, we have
$$\langle \boldsymbol{K}_{-\nu} \boldsymbol{h},\boldsymbol{\phi}_\ell \rangle=(-1)^\ell \Sum_{k=0}^\ell \binom{-\nu}{k} \binom{-1/2}{\ell-k}\log^{-1}\big((\ell-k)\vee2\big)$$
Moreover, since the product $\binom{-\nu}{k} \binom{-1/2}{\ell-k}$ has a constant sign for all $k\leq \ell$, we derive
$$ \log^{-1}(\ell\vee 2) \Big|\Sum_{k=0}^\ell \binom{-\nu}{k} \binom{-1/2}{\ell-k}\Big| \leq |\langle \boldsymbol{K}_{-\nu} \boldsymbol{h},\boldsymbol{\phi}_\ell \rangle |\leq \Big|\Sum_{k=0}^\ell \binom{-\nu}{k} \binom{-1/2}{\ell-k}\Big| $$
but $(-1)^\ell\Sum_{k=0}^\ell \binom{-\nu}{k} \binom{-1/2}{\ell-k}$ is precisely the $\ell^{\text{th}}$ coefficient of the power serie $(1-z)^{-\nu}(1-z)^{-1/2}=(1-z)^{-\nu-1/2}$ which satisfies, thanks to \eqref{Binom equiv},
$$ \Big|(-1)^\ell\Sum_{k=0}^\ell \binom{-\nu}{k} \binom{-1/2}{\ell-k}\Big|\sim \frac{\ell^{\nu-1/2}}{\Gamma(\nu+1/2)} $$
This entails the result.
\end{proof}
\begin{proof}[$\bullet$ Case $\de=0$]
For more clarity, we will suppose that $\boldsymbol{\xi}$ is a white noise (the proof readily adapts otherwise). Let hence $\boldsymbol{K}^0=c_1\boldsymbol{K}_\nu$. Then $\boldsymbol{K}^0\in\mathcal{G}_\nu(Q)$ for an appropriate constant $c_1$, thanks to Proposition \ref{DIP 2 Proposition}. Following the arguments of \citet{Willer}, it suffices to find $\boldsymbol{f}_0$, $\boldsymbol{f}_1$ such that
\begin{enumerate}[label=\roman*)]
\item \label{Ass1} $\boldsymbol{f}_0,\boldsymbol{f}_1\in \mathcal{W}^s(M)$
\item \label{Ass2} $\|\boldsymbol{f}_0-\boldsymbol{f}_1\|^2\gtrsim \ep^{\frac{2s}{s+\nu}}|\log \ep|^{-2}$
\item \label{Ass3} $K(\PP_1,\PP_2)\lesssim 1$ where $\PP_i$ is the law of $\boldsymbol{y}$ under the hypothesis $\boldsymbol{f}_i$, and $K$ is the Kullback-Leibler divergence.
\end{enumerate}
Let $L=c_2\ep^{\frac{-1}{s+\nu}}$. Set $\boldsymbol{f}_0=0$ and define $\boldsymbol{f}_1=c_3\boldsymbol{K}_{-\nu}\boldsymbol{h}$.
\\Point \ref{Ass1}: $\boldsymbol{f}_0$ trivially belongs to the considered set. Moreover, Lemma \ref{Lemma lower bound} entails
\begin{align*}
\|\boldsymbol{f}_1\|_{\mathcal{W}^s}^2\lesssim \ep^2\Sum_{\ell=0}^L \ell^{2s} \ell^{2\nu-1} \lesssim 1
\end{align*}
\\Point \ref{Ass2}: again, thanks to Lemma \ref{Lemma lower bound}, we have
$$\|\boldsymbol{f}_0-\boldsymbol{f}_1\|^2 \gtrsim \ep^2 \Sum_{\ell \leq L} \frac{\ell^{2\nu-1}}{\log^2 (\ell\vee 2)} \gtrsim \ep ^2 L^{2\nu} (\log L)^{-2}\gtrsim \ep^{\frac{2s}{s+\nu}}|\log \ep|^{-2}$$
\\Point \ref{Ass3}: the expression of the Kullback-Leibler divergence in this case is
$$K(\PP_1,\PP_2) = \frac{1}{2} \|\ep^{-1}\boldsymbol{K}^0(\boldsymbol{f}_0-\boldsymbol{f}_1)\|^2=\frac{c_3}{2} \|\boldsymbol{h}\|^2\lesssim 1$$
\\thanks to Lemma \ref{Lemma lower bound}. The choice of appropriate constants $c_i$ clearly yields the result and the proof is complete.
\end{proof}
\begin{proof}[$\bullet$ Case $\ep=0$]
Let $L=c_1 \de^{\frac{-1}{s+\nu}}$. Following the lines of \citet{HR}, we set $\boldsymbol{f}_0=c_2\phi_0$, $\boldsymbol{K}^0=c_3\boldsymbol{K}_\nu$ and we only consider couples $(\boldsymbol{K},\boldsymbol{f})$ such that $\boldsymbol{K}\boldsymbol{f}=\boldsymbol{q}_0$ for a fixed $\boldsymbol{q}_0=\boldsymbol{K}^0\boldsymbol{f}_0$. It is clear that, for well chosen $c_2$ and $c_3$, we have $\boldsymbol{f}_0\in \mathcal{W}^s(M)$ and $\boldsymbol{K}^0\in \mathcal{G}_\nu(Q)$. We thus define $\boldsymbol{H}$ the operator associated to the kernel $\boldsymbol{h}$ and introduce $\boldsymbol{K}^\de=\boldsymbol{K}_\nu+c_4\de\boldsymbol{H}$ a perturbation of $\boldsymbol{K}_\nu$. We shall refer to $\boldsymbol{\dot{g}}^\de$ for the corresponding kernel. Remark that we have 
\begin{align} \label{Lower bound decomp}
\boldsymbol{f}_1-\boldsymbol{f}_0=c_4\de(\boldsymbol{K}^\de)^{-1}\boldsymbol{H}\boldsymbol{f}_0=c_4\de(\boldsymbol{K}^\de)^{-1}\boldsymbol{h}
\end{align}
Furthermore, for $c_4$ small enough, we have thanks to Lemma \ref{Lemma lower bound} and Proposition \ref{DIP 2 Proposition},
$$\|\boldsymbol{K}^\de-\boldsymbol{K}^0\|_{\text{op}}\lesssim \de L^{1/2} \lesssim \de^{\frac{2s+2\nu-1}{2(s+\nu)}}< \frac{1}{2}$$
since $s>1/2$. Hence, the same Neumann serie arguments as in Lemma \ref{Neumann} entail that $\boldsymbol{K}^\de$ belongs to $\mathcal{G}_\nu(Q)$. We now need to check that \ref{Ass1}, \ref{Ass2} and \ref{Ass3} are satisfied, replacing $\ep$ with $\de$.\\
Point \ref{Ass1} : \eqref{Lower bound decomp} and the preceding remark entail 
$$\|\boldsymbol{f}_1-\boldsymbol{f}_0\|_{\mathcal{W}^s}^2 = c_4^2 \de^2\Sum_{\ell\leq L} \ell^{2s} \langle(\boldsymbol{K}^\de)^{-1}\boldsymbol{h}  ,\boldsymbol{\phi}_\ell\rangle \lesssim \de^2 \Sum_{\ell \leq L} \ell^{2s+2\nu-1} \lesssim 1 $$
\\Point \ref{Ass2} : we precise \eqref{Lower bound decomp} and write $$\boldsymbol{f}_1-\boldsymbol{f}_0=c_4 \de\boldsymbol{K}^{-1}\boldsymbol{h}+c_4^2\de^2(\boldsymbol{K}^\de)^{-1}\boldsymbol{K}^{-1}\boldsymbol{H}\boldsymbol{h}$$
Moreover, Lemma \ref{Lemma lower bound} and the preceding remark entail{\small
\begin{center}
$\begin{array}{cccc}
\|\de\boldsymbol{K}^{-1}\boldsymbol{H}\boldsymbol{f}_0\|^2 = \|\de\boldsymbol{K}^{-1}\boldsymbol{h}\|^2 \gtrsim \de^2 \Sum_{\ell \leq L} \frac{\ell^{2\nu-1}}{\log (\ell\vee 2)} \gtrsim \de^2 L^{2\nu} (\log L)^{-2}\gtrsim \de^{\frac{2s}{s+\nu}}|\log \de|^{-2}
\\\|\de^2(\boldsymbol{K}^\de)^{-1}\boldsymbol{K}^{-1}\boldsymbol{H}^2\boldsymbol{f}_0\|^2 \lesssim \de^4 \|(\boldsymbol{K}^\de)^{-1}\|_{\text{HS}} \|\boldsymbol{K}^{-1}\boldsymbol{H}\|_{\text{HS}} \|\boldsymbol{h}\| \lesssim \de^4 L^{4\nu+1} \lesssim \de^{\frac{2s-1}{\nu+s}} \de^2 L^{2\nu}
\end{array}$
\end{center}}
Since $s>1/2$, the second term is negligible with respect to the first. This proves the point \ref{Ass2}.
\\Point \ref{Ass3} : Since we work with couples $(\boldsymbol{K},\boldsymbol{f})$ such that $\boldsymbol{Kf}$ is fixed, we have 
$$K(\PP_0,\PP_1)=\frac{1}{2}\de^{-2} \|\boldsymbol{\dot{g}}^\de-\boldsymbol{\dot{g}}_\nu\|^2=\frac{c_4}{2}\|\boldsymbol{h}\|^2 \lesssim 1$$
thanks to Lemma \ref{Lemma lower bound} and the proof is complete.
\end{proof}
It remains to piece together the two cases $\de=0$ and $\ep=0$ to get the desired result.
\end{proof}
\subsubsection*{Acknowledgements}
The author would like to thank his advisor Dominique Picard for her help and her support.
\bibliographystyle{plainnat}
\bibliography{BiblioLaguerre}

\begin{thebibliography}{28}
\providecommand{\natexlab}[1]{#1}
\providecommand{\url}[1]{\texttt{#1}}
\expandafter\ifx\csname urlstyle\endcsname\relax
  \providecommand{\doi}[1]{doi: #1}\else
  \providecommand{\doi}{doi: \begingroup \urlstyle{rm}\Url}\fi

\bibitem[Abate et~al.(1996)Abate, Choudhury, and Whitt]{ACW}
J.~Abate, G.L. Choudhury, and W.~Whitt.
\newblock On the laguerre method for numerically inverting laplace transforms.
\newblock \emph{INFORMS Journal on Computing}, 8:\penalty0 413--427, 1996.

\bibitem[Abramovich and Silverman(1998)]{AS}
F.~Abramovich and B.~W. Silverman.
\newblock Wavelet decomposition approaches to statistical inverse problems.
\newblock \emph{Biometrika}, 85:\penalty0 115�129, 1998.

\bibitem[Abramovich et~al.(2012)Abramovich, Pensky, and Rozenholc]{APR}
F.~Abramovich, M.~Pensky, and Y.~Rozenholc.
\newblock Laplace deconvolution with noisy observations.
\newblock \emph{ArXiv:1107.2766v2}, 2012.

\bibitem[Ameloot et~al.(1986)Ameloot, Beechem, and Brand]{AB}
M.~Ameloot, J.M. Beechem, and L.~Brand.
\newblock Simultaneous analysis of multiple fluorescence decay curves by
  laplace transforms. deconvolution with reference or excitation profiles.
\newblock \emph{Biophys Chem.}, 23(3-4):\penalty0 155--71, 1986.

\bibitem[Bongioanni and Torrea(2009)]{BT}
B.~Bongioanni and J.~L. Torrea.
\newblock {What is a Sobolev space for the Laguerre function systems?}
\newblock \emph{Studia Mathematica}, 192:\penalty0 147--172, 2009.
\newblock \doi{10.4064/sm192-2-4}.

\bibitem[B\"ottcher and Grudsky(2005)]{BGSiam}
A.~B\"ottcher and S.M. Grudsky.
\newblock \emph{Spectral Properties of Banded Toeplitz Matrices}.
\newblock Siam edition, 2005.

\bibitem[Cinzori and Lamm(2000)]{CL}
A.C. Cinzori and P.K. Lamm.
\newblock Future polynomial regularization of ill-posed volterra equations.
\newblock \emph{SIAM J. Num. Anal.}, 37:\penalty0 949--979, 2000.

\bibitem[Cohen et~al.(2004)Cohen, Hoffmann, and Rei\ss]{CHR}
A.~Cohen, M.~Hoffmann, and M.~Rei\ss.
\newblock Adaptive wavelet {G}alerkin methods for linear inverse problems.
\newblock \emph{SIAM J. Numer. Anal.}, 42:\penalty0 1479--1501, 2004.

\bibitem[Comte et~al.(2012)Comte, Guenod, Pensky, and Rozenholc]{CPR}
F.~Comte, C.~Guenod, M.~Pensky, and Y.~Rozenholc.
\newblock Laplace deconvolution and its application to dynamic contrast
  enhanced imaging.
\newblock \emph{hal-00715943}, 2012.

\bibitem[Delattre et~al.(2012)Delattre, Hoffmann, Picard, and Vareschi]{DHPV}
S.~Delattre, M.~Hoffmann, D.~Picard, and T.~Vareschi.
\newblock Blockwise svd with error in the operator and application to blind
  deconvolution.
\newblock \emph{Electronic Journal of Statistics}, 6:\penalty0 2274--2308,
  2012.

\bibitem[Dey et~al.(1998)Dey, Martin, and Ruymgaart]{DMR}
A.K. Dey, C.F. Martin, and F.H. Ruymgaart.
\newblock Input recovery from noisy output data, using regularized inversion of
  laplace transform.
\newblock \emph{IEEE Trans. Inform. Theory}, 44:\penalty0 1125--1130, 1998.

\bibitem[Donoho(1995)]{Donoho}
D.~Donoho.
\newblock Nonlinear solution of linear inverse problems by wavelet�
  vaguelette decomposition.
\newblock \emph{Appl. Comput. Harmon. Anal.}, 2:\penalty0 101--126, 1995.

\bibitem[Donoho and Johnstone(1994)]{Donoho1994}
D.~Donoho and I~Johnstone.
\newblock Ideal spatial adaptation by wavelet shrinkage.
\newblock \emph{Biometrika}, 81(3):\penalty0 425--455, 1994.

\bibitem[Efromovich and Koltchinskii(2001)]{EK}
S.~Efromovich and V.~Koltchinskii.
\newblock On inverse problems with unknown operators.
\newblock \emph{IEEE Transf. Inf. Theory}, 47:\penalty0 2876--2894, 2001.

\bibitem[Golubev(2010)]{Golubev}
Y.~Golubev.
\newblock On universal oracle inequalities related to high-dimensional linear
  models.
\newblock \emph{AOS}, 38-(5):\penalty0 2751--2780, 2010.

\bibitem[Gradshteyn and Ryzhik(1980)]{GR}
I.S. Gradshteyn and I.M. Ryzhik.
\newblock \emph{Table of Integrals, Series, and Products}.
\newblock Academic press, new york edition, 1980.

\bibitem[Hoffmann and Rei\ss(2008)]{HR}
M.~Hoffmann and M.~Rei\ss.
\newblock Nonlinear estimation for linear inverse problems with error in the
  operator.
\newblock \emph{Ann. Statist.}, 36:\penalty0 310--336, 2008.

\bibitem[Johnstone et~al.(2004)Johnstone, Kerkyacharian, Picard, and
  Raimondo]{JKPR}
I.~Johnstone, G.~Kerkyacharian, D.~Picard, and M.~Raimondo.
\newblock Wavelet deconvolution in a periodic setting.
\newblock \emph{J. R. Stat. Soc. Ser. B Stat. Methodol.}, 66:\penalty0
  547--573, 2004.

\bibitem[Keilson and Nunn(1979)]{KN}
J.~Keilson and W.R. Nunn.
\newblock Laguerre transformation as a tool for the numerical solution of
  integral equations of convolution type.
\newblock \emph{Appl. Math. and Comput.}, 5(4):\penalty0 313--359, 1979.

\bibitem[Lamm(2003)]{Lamm2}
P.K. Lamm.
\newblock Variable-smoothing local regularization methods for first-kind
  integral equations.
\newblock \emph{Inverse problems}, 19:\penalty0 195--216, 2003.

\bibitem[Lien et~al.(2008)Lien, Trong, and Dinh]{LTD}
T.N. Lien, D.D Trong, and A.P.N. Dinh.
\newblock Laguerre polynomials and the inverse laplace transform using discrete
  data.
\newblock \emph{J. Math. Anal. Appl.}, 337:\penalty0 1302--1314, 2008.

\bibitem[Linz(1985)]{Linz}
P.~Linz.
\newblock \emph{Analytical and Numerical Methods for Volterra Equations}.
\newblock Siam edition, 1985.

\bibitem[Mathe and Pereverzev(2003)]{MP}
P.~Mathe and S.V. Pereverzev.
\newblock Geometry of linear ill-posed problems in variable hilbert scales.
\newblock \emph{Inverse problems}, 19:\penalty0 789--803, 2003.

\bibitem[Meckes(2007)]{Meckes}
M.W. Meckes.
\newblock On the spectral norm of a random toeplitz matrix.
\newblock \emph{arXiv:math/0703134v2}, 2007.

\bibitem[Nussbaum and Pereverzev(1999)]{NP}
M.~Nussbaum and S.V. Pereverzev.
\newblock The degree of ill-posedness in stochastic and deterministic models.
\newblock \emph{Preprint No. 509, Weierstrass Institute (WIAS), Berlin}, 1999.

\bibitem[Rathnakumar(1995)]{Rath}
P.K. Rathnakumar.
\newblock A localization theorem for laguerre expansions.
\newblock \emph{Proceedings of the Indian Academy of Sciences - Mathematical
  Sciences}, 105(3):\penalty0 303--314, 1995.

\bibitem[Talagrand(1996)]{Talagrand}
M.~Talagrand.
\newblock Majorizing measures: the generic chaining.
\newblock \emph{Ann. Probab.}, 24(3):\penalty0 1049--1103, 1996.

\bibitem[Willer(2009)]{Willer}
Thomas Willer.
\newblock Optimal bound for inverse problems with jacobi-type eigenfunctions.
\newblock \emph{Statist. Sinica}, 19(2):\penalty0 785--800, 2009.

\end{thebibliography}
\end{document}